\documentclass[10pt,reqno,oneside]{article}
\usepackage{amsmath, amsthm, mathtools, amssymb, enumerate,esint}
\usepackage{graphicx}
\usepackage{xcolor}
\usepackage{caption}
\usepackage{datetime}
\usepackage{fancyhdr}
\usepackage{marginnote}

\usepackage[unicode,breaklinks=true,colorlinks=true,linkcolor=blue,urlcolor=blue,citecolor=blue]{hyperref}
\usepackage[margin=1in]{geometry}
\usepackage{mdframed}
\usepackage{cancel}
\numberwithin{equation}{section}

\begin{document}
\def\tdot{{\gdot}}
\def\intk{[k T_0, (k+1)T_0]}
\def\Dg{{D'g}}
\def\ua{u^{\alpha}}
\def\ques{{\colr \underline{??????}\colb}}
\def\nto#1{{\colC \footnote{\em \colC #1}}}
\def\fractext#1#2{{#1}/{#2}}
\def\fracsm#1#2{{\textstyle{\frac{#1}{#2}}}}   %smaller version of frac
\def\baru{U}
\def\nnonumber{}
\def\palpha{p_{\alpha}}
\def\valpha{v_{\alpha}}
\def\qalpha{q_{\alpha}}
\def\walpha{w_{\alpha}}
\def\falpha{f_{\alpha}}
\def\dalpha{d_{\alpha}}
\def\galpha{g_{\alpha}}
\def\halpha{h_{\alpha}}
\def\psialpha{\psi_{\alpha}}
\def\psibeta{\psi_{\beta}}
\def\betaalpha{\beta_{\alpha}}
\def\gammaalpha{\gamma_{\alpha}}
\def\TTalpha{T_{\alpha}}
\def\TTalphak{T_{\alpha,k}}
\def\falphak{f^{k}_{\alpha}}
\def\R{\mathbb R}
\newcommand {\Dn}[1]{\frac{\partial #1  }{\partial N}}
\def\andand{\text{\indeq and\indeq}}
\def\mm{m}
\def\colr{{}}%%%out
\def\colr{\color{red}}
\def\colu{\color{blue}}
\def\bnew{\color{red}}
\def\enew{\color{black}}
\def\bold{\color{blue}}
\def\eold{\color{black}}
\def\colg{\color{green}}
\def\colb{{}}%%%out
\def\colb{\color{black}}
\def\cole{{}}%%%out
\def\colA{{}}%%%out
\def\colB{{}}%%%out
\def\colC{{}}%%%out
\def\colD{{}}%%%out
\def\colE{{}}%%%out
\def\colF{{}}%%%out

\def\rref#1{{\ref{#1}{\rm \tiny \fbox{\tiny #1}}}}
\def\theequation{\fbox{\bf \thesection.\arabic{equation}}}
\def\ccite#1{{\cite{#1}{\rm \tiny ({#1})}}}
\def\startnewsection#1#2{\newpage\colg \section{#1}\colb\label{#2}}
\setcounter{equation}{0}
\pagestyle{fancy}
\cfoot{}
\rfoot{\thepage}
\chead{}
\rhead{\thepage}
\def\nnewpage{\newpage}
\newcounter{startcurrpage}
\newcounter{currpage}
\def\llll#1{{\rm\tiny\fbox{#1}}}
   \def\blackdot{{\color{red}{\hskip-.0truecm\rule[-1mm]{4mm}{4mm}\hskip.2truecm}}\hskip-.3truecm}%%%out
   \def\bdot{{\color{blue} {\hskip-.0truecm\rule[-1mm]{4mm}{4mm}\hskip.2truecm}}\hskip-.3truecm}%%%out
   \def\purpledot{{\colA{\rule[0mm]{4mm}{4mm}}\colb}}%%%out
   \def\pdot{\purpledot}%%%out
   \def\gdot{{\colB{\rule[0mm]{4mm}{4mm}}\colb}}%%%out

\def\nts#1{{\hbox{\bf ~#1~}}} %nts=note to self
\def\igor#1{{\colu\bf{\hbox{\bf IK: ~#1~}}}} %nts=note to self
\def\qi#1{{\hbox{\bf\color{purple} QX: ~#1~}}} %nts=note to self
\def\wojtek#1{{\hbox{\bf WO: ~#1~}}} %nts=note to self
\def\nts#1{{\colr\small\hbox{\bf ~#1~}}} %nts=note to self%%%out
\def\ntsf#1{\footnote{\colb\hbox{\rm ~#1~}}} %nts=note to self%%%out
\def\bigline#1{~\\\hskip2truecm~~~~{#1}{#1}{#1}{#1}{#1}{#1}{#1}{#1}{#1}{#1}{#1}{#1}{#1}{#1}{#1}{#1}{#1}{#1}{#1}{#1}{#1}\\}%%%out
\def\biglineb{\bigline{$\downarrow\,$ $\downarrow\,$}}%%%out
\def\biglinem{\bigline{---}}%%%out
\def\biglinee{\bigline{$\uparrow\,$ $\uparrow\,$}}%%%out
\def\inon#1{\hbox{\ \ \ \ \ }\hbox{#1}}
\def\onon#1{\inon{on~$#1$}}
\def\inin#1{\inon{in~$#1$}}
\def\Omf{\Omega_{\text f}}
\def\Ome{\Omega_{\text e}}
\def\Gae{\Gamma_{\text e}}
\def\Gaf{\Gamma_{\text f}}
\def\Gac{\Gamma_{\text c}}
\def\wext{\tilde{w}}
\def\wexts{\widetilde{S w}}
\def\wexta{\overline{w}}
\def\wextb{\overline{\overline{w}}}
\def\mbar{{\overline M}}
\def\tilde{\widetilde}

\newtheorem{Theorem}{Theorem}[section]
\newtheorem{Corollary}[Theorem]{Corollary}
\newtheorem{Proposition}[Theorem]{Proposition}
\newtheorem{Lemma}[Theorem]{Lemma}
\newtheorem{Remark}[Theorem]{Remark}
\newtheorem{definition}{Definition}[section]

\def\theequation{\thesection.\arabic{equation}}
\def\endproof{\hfill$\Box$\\}
\def\square{\hfill$\Box$\\}
\def\comma{ {\rm ,\qquad{}} }            %comma in a formula
\def\commaone{ {\rm ,\quad{}} }         %second comma in a formula
\def\dist{\mathop{\rm dist}\nolimits}    %distance
\def\sgn{\mathop{\rm sgn\,}\nolimits}    %sgn
\def\Tr{\mathop{\rm Tr}\nolimits}    %trace
\def\curl{\mathop{\rm curl}\nolimits}    %curl
\def\div{\mathop{\rm div}\nolimits}    %divergence
\def\supp{\mathop{\rm supp}\nolimits}    %divergence
\def\divtwo{\mathop{{\rm div}_2\,}\nolimits}    %two dimensional divergence
\def\re{\mathop{\rm {\mathbb R}e}\nolimits}    %distance
\def\indeq{\qquad{}\!\!\!\!}                     %indentation in formulas
\def\period{.}                           %period in a formula
\def\semicolon{\,;}                      %semicolon in a formula
\newcommand{\cD}{\mathcal{D}}
\newcommand{\eqnb}{\begin{equation}}
\newcommand{\eqne}{\end{equation}}
\newcommand{\na}{\nabla }
\newcommand{\bog}{b}
\newcommand{\bb}{\nu }
\newcommand{\ww}{{\overline{w}}}
\newcommand{\la}{\lambda }
\newcommand{\p}{\partial }
\renewcommand{\d}{\mathrm{d} }
\newcommand{\N}{\mathbb{N}}
\newcommand{\on}{\omega^{(n)}}
\newcommand{\onn}{\omega^{(n+1)}}
\newcommand{\onm}{\omega^{(n-1)}}
\newcommand{\nn}{\mathsf{n}}
\newcommand{\ee}{\mathrm{e}}
\newcommand{\T}{\mathbb{T}}
\renewcommand{\R}{\mathbb{R}}
\newcommand{\lec}{\lesssim  }
\newcommand{\gec}{\gtrsim  }
\newcommand{\tom}{\tilde{\omega}}
\newcommand{\tu}{\tilde{u}}
\newcommand{\un}{^{(n)}}
\newcommand{\unp}{^{(n+1)}}
\newcommand{\unm}{^{(n-1)}}
\newcommand{\lo}{L^2(\Omega)}
\newcommand{\lio}{L^\infty(\Omega)}
\newcommand{\ls}{L^2(S)}
\newcommand{\lis}{L^\infty(S)}
\def\NL{\text{NL}}

\title{Gevrey instability in the inviscid inflow-outflow problem}
\author{Igor~Kukavica, Wojciech~O\.za\'nski, and Qi Xu}

\maketitle
\date{}
\medskip
\begin{abstract}
We consider the 2D incompressible Euler equations on a periodic channel $\mathbb{T}\times (0,1)$ with inflow-outflow boundary condition $u=(0,1)$ on $\mathbb{T} \times \{0,1 \}$. We also impose the incoming vorticity boundary condition $\omega =\eta $ on $\mathbb{T}\times \{ 0 \}$, where $\eta$ is prescribed. We show that the problem is globally well-posed in Gevrey spaces (for any value of the Gevrey exponent $s>1$) as long as $\eta$ remains Gevrey. This proves that the inflow-outflow velocity boundary condition determines the solution locally in time if and only if the solution is considered in an analytic class. In particular, leaving the analytic class, nonuniquness of solutions occurs already in any Gevrey class, by prescribing~$\eta$. 
Hence, prescribing an analytic inflow-outflow velocity leads
to precisely one analytic and continuum $s$-Gevrey solutions for every~$s>1$.
Furthermore, the result implies that if $\eta$ is analytic,
then the unique global solution can lose analyticity in space for all $t>0$, 
but remain $s$-Gevrey regular for all~$s$.

\end{abstract}
%\noindent\thanks{\em Mathematics Subject Classification\/}:
%35R35, %Free boundary problems
%35Q30, %Stokes and Navier-Stokes equations
%76D05  %Navier-Stokes equations 
\noindent\thanks{\em Keywords:\/}
Euler equations, inflow-outflow problem, instability, Gevrey regularity

\section{Introduction}\label{sec_intro}
We consider the two-dimensional incompressible Euler equations,
  \begin{align}
    \begin{split}
        u_t + (u\cdot \nabla) u + \nabla p &= 0,
     \\
     \nabla \cdot u &= 0,
  \end{split}
   \label{EQ01}
  \end{align}
  where $u=(u_1,u_2)$ denotes the velocity field and $p$ is the pressure function. We consider the inflow-outflow problem in a periodic channel
  \begin{equation}
   \Omega
   \coloneqq 
   \mathbb{T}\times [0,1]
   ,
   \label{EQ02}
  \end{equation}
with the size of the torus $\T$ set to~$1$, and the inflow-outflow velocity is given on the boundary,
  \begin{equation}
   u \cdot \mathsf{n} = \overline{u}\cdot \mathsf{n} \qquad \text{ on } \p \Omega ,
   \label{EQ03pre}
  \end{equation}
  where $\overline{u}$ is a prescribed velocity field. 
  
  The problem \eqref{EQ01}--\eqref{EQ03pre} was initially studied by Zaj\c{a}czkowski \cite{Z1,Z2,Z3,Z4}, Antontsev, Kazhikov, and Monakhov~\cite{AKM1,AKM2},
 then by Petcu in~\cite{P}, 
and more recently by Gie, Kelliher, and Mazzucato
 in~\cite{GKM1,GKM2}.
It is known that the inviscid inflow-outflow problem on any domain and in any dimension, becomes very difficult to study as soon as $\overline{u} \ne 0$, see \cite{KOS} for a discussion on the difficulties.  As mentioned in  \cite{GKM2}, the elliptic system one obtains for the pressure function $p$ (see~\cite[(1.2)]{KOS}) and the Gromeka-Lamb form of the Euler equations (see~\cite[(1.6)]{KOS}) give, roughly speaking, 
\eqnb\label{issue}
\text{two relations among four quantities }\omega, \nabla_{\p \Omega}p, u\cdot \mathsf{n} , u - u \cdot \mathsf{n}
\eqne
on the boundary~$\p \Omega$. This  suggests that, in order to obtain any local well-posedness of \eqref{EQ01}--\eqref{EQ07} one must find a way of determining two relations, so that \eqref{issue} provides the remaining two.
This is related to the fact that, in the 3D case, on the inflow portion of $\p \Omega$, the incoming vorticity $\omega$ must lie in the range of $\mathrm{curl}$ and that $\mathrm{div}\,\omega$ must remain conserved (see \cite[(1.6)]{GKM2}). One way of dealing with this issue is to impose an additional boundary condition on the entire velocity field $u$ on the inflow portion $\p \Omega_{\rm in}$ of the boundary~$\p \Omega$. This problem was considered in \cite[Theorem~1.2]{GKM2}, who proved the local well-posedness in $C^{N+1,\alpha}$ ($\alpha \in (0,1)$), provided a certain compatibility condition, involving the $N$-th and ($N+1$)-st time derivatives, holds on $\p \Omega_{\rm in}$ at~$t=0$. Remarkably, as shown by Kukavica, O\.za\'nski, and Sammartino \cite{KOS}, the issue \eqref{issue} disappears for analytic solutions to the inviscid inflow-outflow system if $\overline{u}$ is analytic in space.  \\

\begin{Theorem}[Local well-posedness in analytic spaces \cite{KOS}]\label{T_an}
Consider the analytic norm \[ \| f \|_{X(\tau )} \coloneqq \| f \|_{H^3} + \sum_{| \alpha | \geq  3} \frac{|\alpha |^r }{|\alpha |!} \tau^{|\alpha | - 3 } \epsilon^{\alpha_2} \| D^\alpha u \|_{L^2}\]
 on the domain $\Omega = \T\times (0,1)$, where $\epsilon>0$ is a sufficiently small constant. If $\| \overline{u } \|, \| u_0 \|_{X(\tau_0)} < \infty$ for some $\tau_0 \in (0,1]$, where $u_0$ is a divergence-free initial velocity satisfying \eqref{EQ03pre}, then there exists $T_0, M>0$ such that there is a unique solution $u\in C([0,T_0 ] ; X(\tau(t)))$ to \eqref{EQ01}--\eqref{EQ03pre} with $u(0)=u_0$, where $\tau(t) \coloneqq \tau_0 - Mt$.
\end{Theorem}

We refer the reader to \cite[Theorem~1.2]{KOS} for a more precise statement and the proof. We emphasize that Theorem~\ref{T_an} gives local well-posedness provided \emph{only} the inflow-outflow velocity $\overline{u}$ is prescribed, and it does not require \emph{any} compatibility conditions. The main purpose of this paper is to show that, by replacing the analytic class by any Gevrey class with Gevrey exponent $s>1$, the inflow-outflow velocity $\overline{u}$ is insufficient for local well-posedness.\\

To achieve this, we consider another way of dealing with the issue \eqref{issue}, namely by prescribing the vorticity  $\omega = \mathrm{curl}\,u = \p_1 u_2 - \p_2 u_1$ on the inflow part of the boundary  $\p \Omega$---the so-called \emph{vorticity boundary conditions}. The first result in this direction was established by Zaj\c{a}czkowski~\cite{Z1}, who proved local well-posedness in Sobolev spaces on $3$D domain. Recently~\cite[Theorem~1.3]{GKM2} established local well-posedness in $C^{k,\alpha}$ spaces on $3$D domains. We also refer the reader to \cite{Z2} where a boundary condition for pressure is assumed on the outflow part of the boundary, to \cite{Z2,Z3,Z4} for results concerned with domains with corners, to \cite{GKM2} for an extensive introduction to the problem, as well as to the classical work \cite{AKM2} on the subject.

For simplicity, we consider the $2$D setting as defined in \eqref{EQ02}, and we assume that the inflow-outflow velocity $\overline{u}$ is constant and equal to $(0,1)$,  so that \eqref{EQ03pre} becomes 
 \begin{equation}
   u_2\equiv 1
   \onon{S\cup \tilde S}  = \p \Omega 
   ,
   \label{EQ03}
  \end{equation}
where $S:=\{x_2=0\}$ and $\tilde S:=\{x_2=1\}$.
More general settings can be considered using the methods introduced here. We impose the incoming vorticity boundary condition,
  \begin{equation}
   \omega|_{S}
   = \eta
   ,
   \label{EQ04}
  \end{equation}
where $\eta = \eta (x_1,t)$ is given. Note that the vorticity $\omega$ satisfies the transport equation
\begin{equation}
   \omega_t
   + (u\cdot \nabla) \omega = 0
   .
   \label{EQ07}
  \end{equation}
Observe that $u$ can be  partially  recovered from $\omega$ as a solution of the system 
  \begin{align}
  \begin{split}
   &\curl u = \omega,
   \\&
   \div u = 0,
   \\&
   u_2|_{S\cup \tilde S}= 1
   .
  \end{split}
   \label{EQ13}
  \end{align}
 In order to determine $u$ from \eqref{EQ13}, note that if $u $ solves the Euler equations  \eqref{EQ01}, then 
 \begin{align}
\begin{split}
    \frac{\d}{\d t}\int_{\Omega}u_1
    &=\int_{\Omega}\partial_tu_1 
    =\int_{\Omega}(-u_1\partial_1u_1-u_2\partial_2u_1-\partial_1 p_1)=-\int_{\Omega}u_2\partial_2 u_1
    \\&    =\int_{\Omega} u_1\underbrace{\partial_2u_2}_{=-\partial_1 u_1}+\int_{S}u_1
    -\int_{\tilde S}u_1
    =\underbrace{\int_{\Omega}\partial_1u_2 }_{=0} -\int_\Omega \partial_2u_1
    =\int_{\Omega}\omega,
    \end{split}
    \label{EQ63}
\end{align}
so that $u_1$, $u_2$ can be obtained as solutions of the Poisson problems 
 \begin{align}
  \begin{split}
   \Delta u_1 &= - \partial_{2}\omega \qquad \text{ in }\Omega,
   \\
   \partial_{2}u_1|_{S\cup \tilde S}
   &= - \omega \qquad \text{ on } S \cup \tilde S ,\\
   \int_\Omega u_1 &= \int_0^t \int_\Omega \omega 
  \end{split}
   \label{EQ14}
  \end{align}
and
  \begin{align}
  \begin{split}
   \Delta (u_2-1) &= \partial_{1}\omega \qquad \text{ in } \Omega ,
   \\
   u_2-1& = 0 \qquad \text{ on } S \cup \tilde S, 
  \end{split}
   \label{EQ15}
  \end{align}
respectively. For each $t$, we will denote by 
\eqnb\label{EQ15a}
u[\omega ] \coloneqq (u_1,u_2)
\eqne
the velocity field which solves the systems \eqref{EQ14} and~\eqref{EQ15}. 

We note in passing that  we should not  assign vorticity at the outflow boundary. This is visible already at the level of Sobolev estimates: Note that
  \begin{align}
  \begin{split}
   &  -\int u \cdot \nabla  D^{\alpha} \omega \,D^{\alpha}\omega
  =  \frac12 \int_{S}  (D^{\alpha}\omega)^2 u_2
           - \frac12 \int_{\tilde S}  (D^{\alpha}\omega)^2 u_2
  \leq   \int_{S} |D^{\alpha}\omega|^2
  ,
  \end{split}
   \label{EQ17}
  \end{align}
where we used \eqref{EQ03} in the last step and noted that the integral over $\tilde S$ gives a nonpositive term. Thus, taking $D^\alpha $ of the vorticity equation~\eqref{EQ07}, multiplying by $D^\alpha \omega$ and integrating by parts, we obtain
 \begin{align}
  \begin{split}
  \frac12 \frac{\d}{\d t}
  (D^{\alpha} \omega)^2
  &=
  \int_{S} (D^{\alpha}\omega)^2\underbrace{- \int_{\widetilde{S}} (D^{\alpha}\omega)^2}_{\leq 0}
  -\int \bigl(
          D^{\alpha}(u\nabla \omega)
	  - u \nabla D^{\alpha}\omega 
        \bigr)
	D^{\alpha}\omega 
  ,
  \end{split}
   \label{EQ18}
  \end{align}
which shows that $D^\alpha \omega|_{\widetilde{S}}$ does not increase any Sobolev norm of~$\omega$. We note, however, that the Sobolev existence follows by more extensive considerations than just using the estimate~\eqref{EQ18}, see Section~\ref{sec_sobolev} for details.
\\

Our main result is concerned with $\eta$ belonging to a Gevrey class. Namely, we assume that there exist $N_0(t) ,N(t) \geq 1$ and $s>1$ such that
\begin{align}
        \begin{split}
         \| D^\alpha \eta \|_{L^2(S)}  \leq N_0 (t) N(t)^{|\alpha|}   |\alpha|!^s =: N_{|\alpha |} (t)
    \comma    \alpha	  \in\mathbb{N}_0 \times\{0\}\times \mathbb{N}_0
        \end{split}
        \label{EQ140}
    \end{align}
    and all $t\in [0,T_0]$, where we denoted by $|\alpha|=\alpha_1+\alpha_2+\alpha_3$ the order of a multiindex $\alpha=(\alpha_1,\alpha_2,\alpha_3)$. We also suppose that $u_0$ is such that 
\eqnb\label{EQ140a}
 \| D^\alpha \omega_0 \|_{L^2 (\Omega )} \leq N_0 (0) N(0)^{|\alpha |} | \alpha |!^s 
\eqne    
for all $\alpha \in \N_0^3$, where $N_0(0),N(0)>0$ are constants. We say that $\omega_0\in G_s$ ($s\geq 1$) if \eqref{EQ140a} holds  for some $N_0(0),N(0)
>0$.\\

For convenience, we will abuse the notation slightly by writing
$\eta (\cdot ,0 ) = \omega_0$. Namely, at $t=0$, $\eta$ is a function  $\omega_0$  on $\Omega$, rather than on~$S$.     We will denote by $J_k(t)$ any constant which may depend on $N_i(s)$ for $i\in \{ 0 , \ldots , k\}$ and $s\in [0,t]$. We will often use the short-hand notation $J_k \equiv J_k(t)$. We now state our main result.

\begin{Theorem}[Global well-posedness in Gevrey spaces]
\label{T03}
Let $s>1$. If $\eta $ satisfies \eqref{EQ140}, then there exist 
$L_0(t) ,L(t) \geq 1$,
depending only on $N_0, N$,
and a unique solution $u$ to \eqref{EQ01}, \eqref{EQ03}--\eqref{EQ63} for all $t>0$  such that  $\omega \in C ([0,\infty );H^k )$ for each $k\geq 0$, and 
\begin{align}
     \|D^\alpha \omega\|_{L^2(\Omega)}\leq
     L_0(t) L(t)^{|\alpha|} |\alpha|!^{s}
         \comma t\geq 0
	 \commaone \alpha\in\mathbb{N}_0^3.
        \label{EQ141}
    \end{align}
\end{Theorem}
Now we describe some implications of the main result. The first consequence is that the
velocity inflow-outflow problem  is ill-posed in Gevrey spaces in the following sense.

\begin{Corollary}[Non-uniqueness in Gevrey spaces of \eqref{EQ01}--\eqref{EQ03pre}]\label{cor0}
The inviscid inflow-outflow problem \eqref{EQ01}--\eqref{EQ03pre} (that is if only the normal velocity is prescribed on $\p \Omega$) is ill-posed in each  $s$-Gevrey space, where $s>1$, in the sense that for every $u_0 \in G_s$ there exist infinitely many solutions $u$ with $u(t)\in G_s$ for each $t>0$.
\end{Corollary}
Indeed, Corollary~\ref{cor0} follows simply by prescribing arbitrary $\eta$ satisfying~\eqref{EQ140}. In contrast to Theorem~\ref{T03}, it is interesting that the non-uniqueness occurs already in any Gevrey space $G_s$, no matter how close $s$ is to~$1$. Hence, the term ``Gevrey instability'' in the title.

\begin{Corollary}[Analytic data gives a non-analytic solution]\label{cor1}
Suppose that $\omega_0=0$ and that, for some $T>0$,  $\eta (t)  \ne 0$ is analytic in space for each $t\in (0,T]$, i.e., it satisfies \eqref{EQ140} with $s=1$ for $t\in [0,T]$. Then the  unique  solution obtained by Theorem~\ref{T03} (in any Gevrey class) is not analytic for any $t>0$.
\end{Corollary}
Indeed, by Theorem~\ref{T_an}, the only analytic solution is $u=(0,1)$, which differs, at each $t>0$, from the solution given by Theorem~\ref{T03} (as nonzero $\eta(t)$ is being injected into $\Omega$ at each time $t>0$). 
%\textbf{\colr QX: why do we have $u=0$? I thought we have $u=(0,1)$?}

\section{Proof of Theorem~\ref{T03}}\label{sec_sketch}
Here we prove Theorem~\ref{T03} modulo some claims which we verify in Sections~\ref{sec_vel_ests}--\ref{sec_nonlin}.  For brevity, we use the notation
\[
\| \cdot \|_{p} \coloneqq \| \cdot \|_{L^p (\Omega )}, \qquad \| \cdot \| \coloneqq \| \cdot \|_{2}
\]
throughout the paper.\\

\noindent\texttt{Step 1.}  We define the boundary  and interior Gevrey coefficients.\\

Namely, given   $\epsilon_1,\epsilon_2,\epsilon_3 \in(0,1]$, the initial Gevrey radius $\tau_0\in (0,1]$, a decreasing Gevrey radius function $\tau \colon [0,\infty ) \to (0,1]$ with $\tau(0)=\tau_0$, $\alpha \in \N_0^3$, and $m\in \N_0$, we  define by 
 \eqnb
   a_\alpha (t)\coloneqq \frac{\epsilon^\alpha}{|\alpha|!^s}\tau^{(|\alpha|-2)_+}\|D^\alpha\omega\|_{L^2(S)}
   \label{EQ43}
 \eqne 
 the Gevrey boundary coefficients, and by 
 \begin{align}
    b_\alpha (t) &\coloneqq \frac{\epsilon^\alpha}{|\alpha|!^s}\tau^{(|\alpha|-2)_+}\|D^\alpha\omega\|
   \label{EQ44}
    \end{align}
 the interior Gevrey coefficients of~$\omega$. Here we used the short-hand notation 
 \[ \epsilon \coloneqq (\epsilon_1,\epsilon_2,\epsilon_3) ,\qquad \epsilon^{\alpha}\coloneqq \epsilon_1^{\alpha_1}\epsilon_2^{\alpha_2}\epsilon_3^{\alpha_3} \quad \text{ for }\alpha\in\mathbb{N}_0^{3}.\]

 We will also set
\eqnb
a_m (t) \coloneqq \max_{|\alpha | \leq m } a_\alpha (t), \qquad b_m (t) \coloneqq \max_{|\alpha | \leq m } b_\alpha (t)
\eqne  
for $m\in \N_0$ and $t\geq 0$.\\  

We emphasize  that the boundary coefficients $a_\alpha$, which are motivated by the a~priori estimate \eqref{EQ18}, are evaluated only on the lower boundary $S$, namely the one on which the incoming vorticity $\eta$ is prescribed (recall~\eqref{EQ04}). \smallskip\\

\noindent\texttt{Step 2.} In Section~\ref{sec_vel_ests}, we show  the
order-reduction estimates for $u$, which read
 \begin{align}
    \|D^{(\alpha_1,0,\alpha_3)}u\|
    &\leq
    2\|D^{(\alpha_1-1,0,\alpha_3)}\omega\|\comma \alpha_1-1,\alpha_3\geq 0,
    \label{EQ47}
    \\
    \|D^{(\alpha_1,\alpha_2,\alpha_3)}u\|
    &\leq
    2\sum_{k=1}^{\alpha_2}\|D^{(\alpha_1+k-1,\alpha_2-k,\alpha_3)}\omega\|,
    \quad\alpha_1,\alpha_2- 1,\alpha_3\geq 0.
    \label{EQ48}
    \end{align}
We note in passing that the case of $\alpha_1=\alpha_2=0$ can be treated directly. Indeed, the definition \eqref{EQ14}--\eqref{EQ15} of $u[\omega ]$ and the Biot-Savart inequality in $\dot H^1$ (see \eqref{Hkdot} below) gives 
\eqnb\label{EQ70}
\| \p_t^m u \| \lec \| \p_t^m \omega \|
    \comma m\geq 1.
\eqne
Note that there is no dependence of constants on~$m$ in the last inequality. Moreover, if $D^\alpha$ includes at least one spatial derivative, we have a better estimate. Indeed,  multiplying \eqref{EQ47} and \eqref{EQ48} by $\epsilon^\alpha \tau^{|\alpha | -2} |\alpha |!^{-s}$ gives 
\[
    \begin{split}
        \frac{\epsilon^\alpha\tau^{|\alpha|-2}}{|\alpha|!^s}\|D^\alpha u\|
        &\leq
        \sum_{k=1}^{\alpha_2}\frac{\epsilon^{\alpha+(k-1)e_1-ke_2}\tau^{|\alpha|-1-2}}{(|\alpha|-1)!^s}
        \|D^{(\alpha_1+k-1,\alpha_2-k,\alpha_3)}\omega\|
        \left(\frac{\epsilon_2}{\epsilon_1}\right)^k\epsilon_1\tau|\alpha|^{-s}
        \\&
       \leq
        \sum_{k=1}^{\alpha_2}\left(\frac{\epsilon_2}{\epsilon_1}\right)^k\epsilon_1\tau|\alpha|^{-s}
        \max\limits_{|\beta|\leq|\alpha|-1}b_\beta
    \end{split}
   \]
for $\alpha_1+\alpha_2\geq1$ and $|\alpha|\geq 3$, which, recalling that $\epsilon_2\leq \epsilon_1/2$, gives a very useful inequality,
\begin{align}
    \begin{split}
        \frac{\epsilon^\alpha\tau^{|\alpha|-2}}{|\alpha|!^s}\|D^\alpha u\|
        &
        \lec
	\frac{\tau}{|\alpha|^{s}}
	\max\limits_{|\beta|\leq|\alpha|-1}b_\beta
    \end{split}
    \label{EQ67}
\end{align}
for $\alpha_1+\alpha_2\geq1$ and $|\alpha|\geq 3$.\\

\noindent\texttt{Step 3.} We prove local well-posedness in Sobolev spaces.\\

Namely, in Section~\ref{sec_sobolev} we show the following. 
\cole
\begin{Theorem}
\label{T01}
Let $k\in\{2,3,\ldots\}$, and suppose that $\eta$ satisfies \eqref{EQ140} for all $\alpha$ such that $| \alpha | \leq k$. Then there exists a unique solution $\omega \in C ([0,\infty ); H^k (\Omega ))$ to \eqref{EQ07} with $u=u[\omega ]$ and $\| \omega \|_{H^k} \leq J_k (t)$ for all $t\geq 0$. 
\end{Theorem}
\colb

%The theorem is established in Section~\ref{sec_sobolev} below.
Using Theorem~\ref{T01}
with $k=3$, we obtain that there exists a continuous increasing function  $K(t)=J_2(t)$ such that $K(t) \geq 1+ N_0 (t) N(t)^2$, 
\eqnb\label{EQ171}
a_2 (t) , b_2 (t) \leq K(t) 
,
\eqne
and 
\eqnb \label{EQ171a}
\| u [\omega ] \|_{\infty}, \| u [\p_t \omega ] \|_{\infty}, \| \omega \|_{H^2} \leq K(t) 
\eqne
for all $t\geq 0$. We fix such~$K(t)$. In the following we will write $K \equiv K(t)$, for brevity. 
\smallskip\\
We note that, as compared to Step~3, the claim of Theorem~\ref{T03} requires much more careful control of the growth of the $L^2$ norms of derivatives of~$\omega$. Indeed, for the global well-posedness in $H^k$ we use (in Section~\ref{sec_sobolev}) rather crude estimates of the form $\| D^\alpha \omega \| \leq J_k$ (see \eqref{L2Sb}), where $J_k(t)$ is an unquantified  upper bound depending on the first $k$ derivatives of~$\eta$. Thus, in the following steps, we obtain a more precise control on the coefficients $a_k$ and~$b_k$, which allows us to control the Gevrey norm of~$\omega$.\\

\noindent\texttt{Step~4.} We observe that for $|\alpha | \geq 3$ we can also estimate $a_{\alpha }$ in the case when $\alpha_2 =0$, provided that
\eqnb\label{tau_smallness}
\tau(t) \leq \frac{1}{N(t)}
    \comma t\geq 0.
\eqne
Namely, given the condition \eqref{tau_smallness},  the assumption \eqref{EQ140} implies
\eqnb\label{a_K_bound}
a_\alpha (t) \leq K(t)
%\qquad \text{ for all }
    \comma \alpha\in \N_0 \times \{ 0 \} \times \N_0
    \commaone t\geq 0.
\eqne
\smallskip\\
For controlling the quantities $a_{\alpha}$ and $b_\alpha$'s in full generality, we first fix  $\epsilon_1,\epsilon_3 \coloneqq 1$ and  
\eqnb\label{eps_rel}
\epsilon_2 (t) \coloneqq \frac{1}{4K(t)}.
\eqne
\smallskip\\

\noindent\texttt{Step~5.} We show in  Section~\ref{sec_bdry_coeffs} that 
\eqnb\label{bdry_coefs}
b_m (t) \leq K (t) \qquad \text{ implies } \qquad a_{m+1}(t) \leq K (t)
\eqne
for each $t\geq 0$ such that
\eqnb\label{tau_small1}
\tau (t) \leq \frac{1}{C_s  K(t)^3},
\eqne
where $C_s>0$ is a universal constant (defined in \eqref{Cs_def}), which depends only on the Gevrey exponent $s>1$.\\
In order to estimate the time evolution of the coefficients~$b_m$, we use~\eqref{EQ18} to obtain
  \begin{align}
  \begin{split}
  \frac12 \frac{\d}{\d t} b_\alpha^2 &= \frac12 \frac{\d}{\d t}
  \left(
    \frac{\epsilon^{2\alpha}}{|\alpha|!^{2s}}
    \tau^{2|\alpha|-4}
    \int (D^{\alpha} \omega)^2
  \right)
  \\
  &=\underbrace{\alpha_2 \frac{\dot \epsilon_2}{\epsilon_2} b_{\alpha }^2 }_{\leq 0} + (|\alpha|-2) \frac{\dot \tau}{\tau } b_\alpha^2 +  \frac{\epsilon^{2\alpha}}{|\alpha|!^{2s}}
    \tau^{2|\alpha|-4}
  \frac{\d }{\d t}  \int (D^{\alpha} \omega)^2
  \\&
  \leq
 \frac{|\alpha|}4 \frac{\dot \tau}{\tau } b_\alpha^2  + \underbrace{ \frac{\epsilon^{2\alpha}}{|\alpha|!^{2s}}
  \tau^{2|\alpha|-4}
  \int_{S} (D^{\alpha}\omega)^2}_{=a_\alpha^2}
  -
  \frac{\epsilon^{2\alpha}}{|\alpha|!^{2s}}
  \tau^{2|\alpha|-4}
  \int \bigl(
          D^{\alpha}(u\cdot \nabla \omega)
	  - u \cdot \nabla D^{\alpha}\omega 
        \bigr)
	D^{\alpha}\omega \\
	&=:  \frac{|\alpha|}4  \frac{\dot \tau}{\tau } b_\alpha^2  + a_\alpha^2 + \NL (\alpha )
  \end{split}
   \label{EQ19}
  \end{align}
for each $\alpha$ such that $|\alpha |\geq 3$.
\smallskip\\

\noindent\texttt{Step~6.} We show in Section~\ref{sec_nonlin} that there exists a constant $\widetilde{C}\geq 1$ such that if $b_m(t) \leq K(t)$ then
\eqnb\label{nonlin_est}
\NL (\alpha ) \leq \widetilde{C} m  K(t)^4 b_{m+1} (t)^2
\eqne
for each $t>0$, $m\geq 2$, and all multiindices $\alpha$ such that $| \alpha |= m+1$.\smallskip\\

We now choose the time dependence of the Gevrey radius as 
\eqnb\label{tau_def}
\tau (t) \coloneqq \tau (0) \ee^{-8\widetilde{C}\int_0^t  K^{6} (s) \d s}
\eqne
with $\tau (0) \in (0,1]$ sufficiently small so that \eqref{tau_smallness} and \eqref{tau_small1} hold for all $t\geq 0$. \smallskip \\

\noindent\texttt{Step~7.}  We show that $b_m (t)\leq K(t)$ for all $m$ and $t\geq 0$.\\

Indeed,
for $m\leq 2$, the claim follows from \eqref{EQ171}.
For $m \geq 3$ we proceed by induction: First, $b_{m+1}(0)< K$ by assumption \eqref{EQ140a}, since $\tau(0)\leq 1$. Moreover, $b_{m+1} (t)$ is continuously differentiable (a consequence of Theorem~\ref{T01}). Suppose that $b_{m+1} (t) > K(t)$ for some $m\geq 2$, $t>0$, and let 
\[
t_0 \coloneqq \inf \bigl\{ t>0 \colon b_{m+1} (t) > K(t) \bigr\}.
\]
 Let $\alpha$ be any multiindex such that $|\alpha |= m+1$. Applying Steps 5 and 6 in \eqref{EQ19}, we obtain
\begin{align}
\begin{split}
 \frac{\d}{\d t}b_\alpha^2
    &\leq \frac{m}2\frac{\dot\tau}{\tau}b_\alpha^2 + 2 K^2 + 2 \widetilde{C} m K^4 b_{m+1}^2 ,
        \end{split}
\end{align}
where we omitted ``$t$'' in the notation. Recalling \eqref{tau_def}, we have $\dot \tau/\tau \geq - 8 \widetilde{C}  K^6  $, so that, at $t_0$,
\begin{align}
\begin{split}
 \frac{\d}{\d t}b_\alpha^2 (t_0)
    &\leq - 4\widetilde{C} m  K^6 (t_0) + 2 K^2(t_0) + 2\widetilde{C} m  K^6(t_0)
    \leq  -\widetilde{C} m K^6  (t_0) <0 .
    \end{split}
\end{align}
Thus, for each $\alpha$ with $|\alpha |=m+1$, $b_\alpha (t)$ decreases on some time interval following~$t_0$. This contradicts the definition of $t_0$, and so the claim follows. \\

\noindent\texttt{Step 8.} We conclude the proof of Theorem~\ref{T03}.\\

From Step~3 we know that there exists a unique global-in-time solution $\omega \in C([0,\infty );H^k )$ for each $k\geq 0$. By Step~7 we also have that $\omega (t)$ remains Gevrey for all  $t>0$, and the claim~\eqref{EQ141} follows by writing
\[
\| D^\alpha \omega \| = \frac{|\alpha |!^s }{\epsilon_2 (t)^{\alpha_2} \tau(t)^{(|\alpha |-2)_+}} b_\alpha (t) \leq K(t)^{1+|\alpha | } \tau (t)^{-|\alpha |} |\alpha |!^s
\]
and setting $L_0(t) \coloneqq K(t)$, $L(t) \coloneqq K(t) \tau(t)^{-1}$.\bigskip\\

The paper is structured as follows. After introducing some preliminary concepts and combinatorial estimates (for $s>1$) in Section~\ref{sec_prelim}, we discuss velocity estimates in Section~\ref{sec_vel_ests}, and we prove global well-posedness in Sobolev spaces (Theorem~\ref{T01}) in Section~\ref{sec_sobolev}. We then prove \eqref{bdry_coefs} in Section~\ref{sec_bdry_coeffs} and \eqref{nonlin_est} in Section~\ref{sec_nonlin}.

\section{Preliminaries}\label{sec_prelim}

We denote by $C(t)$ a generic function of time, which can change from line to line, and similarly we denote by $C_k (t)$ any such constant which may depend on a parameter~$k$.

\begin{Lemma}[Binomial inequalities]
\label{L03}
Let $s>1$.
%For $n\in\mathbb{N}$, the following inequalities hold
Then we have the inequalities
    \begin{eqnarray}
      \sum_{k=1}^{n-1}
        \binom{n}{k}^{1-s}
        \frac{(k+1)^s(n-k+1)^s}{n}&\lec_s 1
	,\label{EQ29} \\
	 \sum_{k=0}^n\binom{n}{k}\binom{n+1}{k}^{-s}
	&\lec_s 1
	\label{EQ38}        
    \end{eqnarray}
for $n\in \mathbb{N}$.
\end{Lemma}
\colb
Note that, since $n \leq (n+1)^s$, the inequality \eqref{EQ29}  also implies
 \eqnb\label{EQ30}
 \begin{split}
        \sum_{k=0}^{n}
        \binom{n}{k}^{1-s}
        \frac{(k+1)^s(n-k+1)^s}{(n+1)^s} &= 2+ \sum_{k=1}^{n-1}
        \binom{n}{k}^{1-s}
        \frac{(k+1)^s(n-k+1)^s}{(n+1)^s}\\        
	&\leq 2+ \sum_{k=1}^{n-1}
        \binom{n}{k}^{1-s}
        \frac{(k+1)^s(n-k+1)^s}{n}
	\lec_s 1.
        \end{split}
        \eqne

\begin{proof}[Proof of Lemma~\ref{L03}]
For \eqref{EQ29}, denote by
\begin{align}
    \begin{split}
        T_{n,k}&\coloneqq \frac{1}{n}\binom{n}{k}^{1-s}(k+1)^s(n-k+1)^s
       \end{split}
\end{align}
the summand  on the left side of~\eqref{EQ29}.
Since $T_{n,k}=T_{n,n-k}$,
we only need to estimate the sum
from $k=1$ to $\lfloor n/2\rfloor$ and assume
  \begin{equation}
   2k\leq n
   .
   \label{EQ32}
  \end{equation}
First, since 
$k\leq n-1$, we have  $
    (n-j)k\geq nk-jn=n(k-j)
$ for every $j\in \{ 0,\ldots , k-1 \}$, and so
\begin{align}
    \frac{n-j}{k-j}\geq\frac{n}{k}
  \end{align}
for such~$j$. Therefore, 
\[
    \binom{n}{k}=\prod_{j=0}^{k-1}\frac{n-j}{k-j}\geq\left(\frac{n}{k}\right)^k,
   \]
and thus
\begin{align}
  \begin{split}
    T_{n,k}
    &\leq\frac{1}{n}\left(\frac{n}{k}\right)^{k(1-s)}(k+1)^s(n-k+1)^s
    \leq
    \frac{1}{n}\left(\frac{n}{k}\right)^{k(1-s)}(2k)^sn^s
    \\&
    =2^s n^{(k-1)(1-s)}k^{-k(1-s)+s}\\
    &\leq 2^s (2k)^{(k-1)(1-s)}k^{-k(1-s)+s}
     =
     2^{s+(k-1)(1-s)}
     k^{2s-1}
  \end{split}
    \label{EQ35}
\end{align}
for each $k\in \{ 1, \ldots , \lfloor n/2 \rfloor \}$, where we used \eqref{EQ32} in the last inequality. Hence,
\begin{align}
    \sum_{k=1}^{\lfloor n/2\rfloor}T_{n,k}
    \leq
    \sum_{k=1}^{\lfloor n/2\rfloor}
    2^{s+(k-1)(1-s)}
    k^{2s-1}
    \leq
    \sum_{k=1}^{\infty}
    2^{s+(k-1)(1-s)}
    k^{2s-1}
    \lec_s 1,
\end{align}
where the last inequality holds 
since
$s+(k-1)(1-s)<0$ provided $k$ is
larger than a constant depending on~$s$ only.\\

 For \eqref{EQ38} we first observe that 
\[
    \sum_{k=0}^n\binom{n}{k}\binom{n+1}{k}^{-s}
=\sum_{k=0}^n\binom{n}{k}^{1-s}\left(\frac{n+1}{n-k+1}\right)^{-s}
  \leq \sum_{k=0}^n \binom{n}{k}^{1-s}
  = 2+\sum_{k=1}^{n-1}\binom{n}{k}^{1-s}.
    \]
Using the symmetry of the combinatorial coefficients, we obtain
\begin{align}
    \sum_{k=1}^{n-1}\binom{n}{k}^{1-s}
    \leq 2\sum_{k=1}^{\left\lfloor \frac{n}{2} \right\rfloor}\binom{n}{k}^{1-s}
    \leq 2\sum_{k=1}^{\left\lfloor \frac{n}{2} \right\rfloor}\binom{2k}{k}^{1-s}
    \leq 2\sum_{k=1}^\infty\binom{2k}{k}^{1-s}.
    \label{EQ40}
\end{align}
Since $\binom{2k}{k}
    =\max_{i}\binom{2k}{i}$ for each $k\geq 1$, we have 
\begin{align*}
    \binom{2k}{k}
    \geq\frac{1}{2k+1}\sum_{i=0}^{2k}\binom{2k}{i}
    =\frac{4^k}{2k+1}
    \geq\frac{4^{k-1}}{k}
    ,
\end{align*}
so that 
\begin{align}
    \sum_{k=1}^\infty\binom{2k}{k}^{1-s}
    \leq 2+2\sum_{k=1}^\infty \left(\frac{4^{k-1}}{k}\right)^{1-s}\lec_s 1. 
   \end{align}
Applying this inequality in  \eqref{EQ40} concludes the proof. 
\end{proof}

We will often use the  combinatorial identity 
\begin{align} \sum_{0\leq\beta\leq\alpha,|\beta|=k}\binom{\alpha}{\beta}=\binom{m}{k}
    \label{EQ77}
\end{align}
for any multiindex $\alpha = (\alpha_1 , \ldots , \alpha_m)$.
It follows by observing that both sides of the equality are equal to the number of possible choices of a $k$-element subset of an $m$-element set.

Finally, we recall from 
\cite[Lemma~2.3]{KNV}
the periodic De~Rham's theorem.

\cole
\begin{Lemma}[Periodic De Rham's theorem]
\label{L05}
Assume that $v \in L^2_{\mathrm{loc}}(\mathbb{R} \times [0,1])$
is $1$-periodic in the $x$ variable, and suppose that it satisfies
  \begin{align}
   \nabla^\perp \cdot v = 0
   \andand
   \int_{[0,1]\times[0,1]} v^1 = 0
   .
   \label{EQ200}
  \end{align}
Then there exists a function $q \in H^1_{\mathrm{loc}}(\mathbb{R} \times [0,1])$, which is $1$-periodic in the $x$ variable and satisfies
\begin{align}
  v = \nabla q
\end{align}
in $\mathbb{R} \times (0,1)$.
\end{Lemma}
\colb

Finally, we prove a log-type inequality on~$\Omega$. To this end we define the BMO norm on $\Omega$, 
\[
\|f\|_{\text{BMO}(\Omega)}\coloneqq|\Omega|^{-1}\|f\|_{L^1(\Omega)}+\sup_{x\in \Omega, r>0}
|\Omega\cap B_r(x)|^{-1}\int_{\Omega\cap B_r(x)}\left|f-\fint_{\Omega\cap B_{r}(x)}f\right|.
\]
\begin{Lemma}
[A log-type estimate on a strip]
\label{L_Log01}
For each $f\in W^{1,4}(\Omega)$, we have
\[
    \|f\|_{\infty}\lec 1+\|f\|_{\text{BMO}(\Omega)}\left(1+\log^+\|\nabla f\|_{4 }\right).
\]
\end{Lemma}
\begin{proof}
    For any $r\in (0,1)$, we define 
    \[
    f_{B'_r(x)}\coloneqq \frac{1}{|B'_r(x)|}\int_{B'_r(x)} f(y)\,\text{d}y,
    \] where $B'_r(x)=\Omega\cap B_r(x)$.
 In order to make sense of $f_{B_r(x)}$, we now identify $f$ with its even extension across $S\cup \widetilde{S}$.  We have
    \[
    \begin{split}
        \left|f(x)-f_{B'_r(x)}\right|&\leq \frac{1}{|B'_r(x)|}\int_{B'_r(x)}|f(x)-f(y)|\,\text{d}y
        \lec    \frac{1}{|B_r(x)|}\int_{B_r(x)}|f(x)-f(y)|\,\text{d}y
        \\& \lec    
        r^{1/2}\|\nabla f\|_{L^4(B_r(x))}
        \lec r^{1/2}\|\nabla f\|_{4}
    \end{split}
    \]
    for each $x\in \Omega $,  where we used Morrey's inequality in the third inequality.     We fix 
    \[
    r\coloneqq \min\left\{\frac{1}{2},(1+\|\nabla f\|_{4 })^{-2}\right\},
    \]
    so that
    \[
        |f(x)-f_{B'_r(x)}|\lec 1
    \]
    for all $x\in \Omega$.
    We also choose dyadic radii $r_j=2^j r$ for all $j\geq 1$, and we let $J\geq 1$ be such that $r_J\geq 1/2$. Note that
\[
J \lec 1+ \log^+ \| \nabla f \|_{4}
\]    
and that the sets $B'_{r_j}(x)$ are  doubling, uniformly with respect to $x\in \Omega $. Thus 
    \begin{align}
        \begin{split}
        &|f_{B'_{r_j}(x)}-f_{B'_{r_{j+1}}(x)}|=\left|\fint_{B'_{r_j}}(f-f_{B'_{r_{j+1}}(x)})\right|
        \leq \fint_{B'_{r_j}}|f-f_{B'_{r_{j+1}}(x)}|
        \\&\indeq\leq 2\fint_{B'_{r_j+1}}|f-f_{B'_{r_{j+1}}(x)}|\lec_{\Omega}\|f\|_{\text{BMO}(\Omega)}
        \end{split}
    \end{align}
for all $j\leq J-1$, $x\in \Omega$. Therefore, 
    \[
    \begin{split}
        &|f(x)|\leq |f_{B'_{r_J}(x)}|+\sum_{j=0}^{J-1}|f_{B'_{r_j}(x)}-f_{B'_{r_{j+1}}(x)}|+|f(x)-f_{B'_r(x)}|
        \\&\lec J\|f\|_{\text{BMO}(\Omega)}+1
        \lec 1+\|f\|_{\text{BMO}(\Omega)}
        \left(1+\log^+(\|\nabla f\|_{4 }\right)
        \end{split}
    \]
    for all $x\in \Omega$, which concludes the proof.
\end{proof}

\section{Velocity Estimates}\label{sec_vel_ests}

In this section we show the order reduction estimates \eqref{EQ47}--\eqref{EQ48}, and prove some Biot-Savart estimates on $u[\omega ]$ (recall \eqref{EQ14}--\eqref{EQ15a} for the definition). We also discuss log-Lipschitz inequalities in Lemma~\ref{L6}.

As for the order reduction estimates \eqref{EQ47}--\eqref{EQ48} for $u$, i.e.,
    \begin{align}
    &
    \|D^{(\alpha_1,0,\alpha_3)}u\|
    \leq
    2\|D^{(\alpha_1-1,0,\alpha_3)}\omega\|\comma \alpha_1-1,\alpha_3\geq 0,
    \label{EQ47copy}
    \\&
    \|D^{(\alpha_1,\alpha_2,\alpha_3)}u\|
    \leq
    2\sum_{k=1}^{\alpha_2}\|D^{(\alpha_1+k-1,\alpha_2-k,\alpha_3)}\omega\|,
    \quad\alpha_1,\alpha_2- 1,\alpha_3\geq 0,
    \label{EQ48copy}
    \end{align}
we note that applying tangential and time differential operators on each side
of the Poisson equation \eqref{EQ14} for $u_1$ preserves the structure, and so we deduce, analogously to
\eqref{EQ55}, that
\begin{align}
    \|D^{(\alpha_1,0,\alpha_3)}\nabla u_1\|
    \leq
    \|D^{(\alpha_1,0,\alpha_3)}\omega\|
    \comma \alpha_1,\alpha_3\geq0    
    ,
    \label{EQ56}
\end{align}
and, analogously to \eqref{EQ57},
  \begin{equation}
    \|
    D^{(\alpha_1,0,\alpha_3)}\nabla u_2
    \|
    \leq
    2\|D^{(\alpha_1,0,\alpha_3)}\omega\|   
    \comma
     \alpha_1,\alpha_3\geq0
     ,
   \label{EQ143}
  \end{equation}
which proves~\eqref{EQ47copy}.
 Since $\p_2 u_1 = \omega - \p_1 u_2$, we have
\begin{align}
    \begin{split}
        &\|D^{(\alpha_1,\alpha_2,\alpha_3)}u_1\|
        =\|D^{(\alpha_1,\alpha_2-1,\alpha_3)}\partial_2u_1\|
        \\&\indeq
        \leq \|D^{(\alpha_1,\alpha_2-1,\alpha_3)}\omega\|
        +\|D^{(\alpha_1+1,\alpha_2-1,\alpha_3)} u_2\|
    \end{split}
    \label{EQ60}
\end{align}
for $\alpha_1,\alpha_2-1,\alpha_3\geq 0$.
Further, the divergence-free condition implies
\begin{align}
        \|D^{(\alpha_1,\alpha_2,\alpha_3)}u_2\|
        =\|D^{(\alpha_1,\alpha_2-1,\alpha_3)}\partial_2u_2\|
        =\|D^{(\alpha_1+1,\alpha_2-1,\alpha_3)}u_1\|
        \label{EQ61}
\end{align}
for $\alpha_1,\alpha_2-1,\alpha_3\geq 0$.
Using~\eqref{EQ60} and~\eqref{EQ61} iteratively, we get
\begin{align}
    \|D^{(\alpha_1,\alpha_2,\alpha_3)}u_i\|
    \leq
    2
    \sum_{k=1}^{\alpha_2}\|D^{(\alpha_1+k-1,\alpha_2-k,\alpha_3)}\omega\|
 \end{align}
for $i=1,2$ and  $\alpha_1,\alpha_2-1,\alpha_3\geq 0$, which completes the proof of~\eqref{EQ48copy}.\\

In order to state the Biot-Savart estimates, we recall that, even though $u[\omega ]$ is defined (below~\eqref{EQ14}--\eqref{EQ15}) for each $t$, it uses $\omega$ for all times between $0$, $t$, via the normalization condition $\int u_1 = \int_0^t \int \omega$ from~\eqref{EQ14}. However, for $D^k  u [\omega ]$, with $k\geq 1$, we expect this normalization condition to play no role in Biot-Savart estimates, which we now show.

\begin{Lemma}[The Biot-Savart estimates]\label{L02}
For all $k\geq 1$ 
\eqnb\label{Hkdot}
\| u [\omega ] \|_{\dot H^k } \lec_k \| \omega \|_{\dot H^{k-1}}.
\eqne
\end{Lemma}
As for the $L^2$-level estimates, we use the normalization condition for $u_1$ from \eqref{EQ14},   the boundary condition for $u_2$ from \eqref{EQ15}, and the Poincar\'e inequality  \eqref{Hkdot} to get 
 \begin{align}
        \|u_1\|&\leq \left\| u_1 - \int_\Omega u_1  \right\| + \left| \int_\Omega u_1 \right| \lec  \|\omega\|+\int_0^t\|\omega\|,
           \label{EQ50}\\
           \|u_2\|&\lec 1+ \| \nabla u_2 \| \lec  \|\omega\|+1 \label{EQ49}
    .
    \end{align}

\begin{proof}[Proof of Lemma~\ref{L02}.]

Testing~$\eqref{EQ14}_1$ with $u_1$ gives
  \begin{align}
    \begin{split}
    -\int_{\Omega}\partial_2\omega u_1
    &=\int_{\Omega}u_1\Delta u_1 
    =
    - \int_{\Omega}|\nabla u_1|^2
    - \int_{\tilde{S}}\omega u_1
    + \colb \int_{S}\omega u_1
%    \\&
%    =
%    \int_\Omega\omega\partial_2u_1-\int_{\tilde{S}}\omega u_1+\int_{S}\omega u_1
    ,
  \end{split}
    \label{EQ52}
  \end{align}
where we used the boundary condition  \eqref{EQ14}\textsubscript2 in the second equality.
On the other hand, integrating by parts yields
\begin{align}
    -\int_\Omega\partial_2\omega u_1
    =\int_\Omega\omega\partial_2u_1-\int_{\tilde{S}}\omega u_1+\int_{S}\omega u_1.
    \label{EQ53}
\end{align}
%Note that $u_1|_{S\cup \tilde S}=0$.
Combining \eqref{EQ52} and \eqref{EQ53}, we obtain
\[
    \int_\Omega |\nabla u_1|^2=-\int\omega\partial_2u_1\leq
    \|\nabla u_1\|_{L^2(\Omega)} \|\omega\|_{L^2(\Omega)}
   ,
   \]
from where
\begin{align}
    \|\nabla u_1\|\leq \|\omega\|.
   \label{EQ55}
\end{align}
Since \eqnb\label{EQ55a}
    \partial_1u_2-\partial_2u_1=\omega \quad \text{ and } \quad \partial_1 u_1+\partial_2 u_2=0,\eqne
    we also derive
    \begin{align}
    \|\nabla u_2\|\lec \|\omega\|.
   \label{EQ57}
\end{align}
Moreover, for each $k\geq 1$, we  use the derivative reduction estimates~\eqref{EQ47copy}--\eqref{EQ48copy}, and the  same argument as \eqref{EQ52}--\eqref{EQ53} with $u$ replaced by $\partial_1^{k}u$, and apply $\p_1^k $ of \eqref{EQ55a} to obtain \eqref{Hkdot}, as required. 
\end{proof}
\begin{Lemma}[The log-Lipschitz estimates]\label{L6}
We have 
\eqnb\label{logW14}
\| \nabla u [\omega ] \|_\infty \lec 1+ \| \omega \|_\infty \log (2+ \| \omega \|_\infty +\| \nabla \omega \|_{4} )
\eqne
and
\eqnb\label{logHk}
\| \nabla u [\omega ] \|_\infty \lec_k 1+ \| \omega \|_\infty \log (2+  \| \omega \|_{H^k} )
\eqne
for every $k\geq 2$.
\end{Lemma}
Lemma~\ref{L6} is a version of the Brezis-Weinger inequality (see~\cite{BG,BW} and \cite{KT}). Since the Biot-Savart law defined by \eqref{EQ14}--\eqref{EQ15} is nonstandard and involves inflow and outflow, we provide the proof for the reader's convenience.  
\begin{proof}
Let $\psi$ be the solution to
    \eqnb\label{EQL01}
     -\Delta\psi=\omega \quad \text{ in } \Omega \comma \psi|_{S\cup\tilde S}=0
    \eqne
    that is periodic in~$x_1$. By the uniqueness of solutions to the elliptic problems \eqref{EQ14}--\eqref{EQ15}, we observe that
    \[
        u_1=\partial_2\psi+\frac{1}{|\Omega|}\int_0^t \int_\Omega \omega
\andand
        u_2=-\partial_1\psi+1.
    \]
Now we focus on estimating the stream function~$\psi$.
    We let $M\coloneqq \mathbb{T}\times(\mathbb{R}/2\mathbb{Z})$, represented as $\mathbb{T}\times(-1,1)$ with $x_2=1$ and $x_2=-1$ identified. We define the odd extensions
    \[
        \psi^{\mathrm{odd}} (x_1,x_2) \coloneqq 
        \begin{cases}
        \psi(x_1,x_2)\comma    &0<x_2<1,
        \\
        -\psi(x_1,-x_2)\comma   &-1<x_2<0,
        \end{cases}
    \]
    and
    \[
        \omega^{\mathrm{odd}}(x_1,x_2) \coloneqq 
        \begin{cases}
        \omega(x_1,x_2)\comma    &0<x_2<1,
        \\
        -\omega(x_1,-x_2)\comma   &-1<x_2<0,
        \end{cases}
    \]
    and we extend both $\psi^{\mathrm{odd}}$ and $\omega^{\mathrm{odd}}$ $2-$periodically in~$x_2$. For any test function
    $\eta\in C_c^\infty(M)$, we define
    \[
    \eta^{\mathrm{odd}}(x_1,x_2)\coloneqq \eta(x_1,x_2)-\eta(x_1,-x_2).
    \]
    Because $\eta(x,1)=\eta(x,-1)$ for $x\in \mathbb{T}$,
    $\eta^{\mathrm{odd}}$ has zero trace on $S\cup\tilde S$. By a change of variable, we obtain
    \[
        \int_{M}\nabla\psi^{\mathrm{odd}}\cdot\nabla\eta
        =\int_\Omega\nabla\psi\cdot\nabla\eta^{\mathrm{odd}}
        =\int_\Omega\omega\eta^{\mathrm{odd}}=\int_M\omega^{\mathrm{odd}}\eta.
    \]
    Therefore, 
    \[
        -\Delta \psi^{\mathrm{odd}}=\omega^{\mathrm{odd}} \onon{M}.
    \]
    By a standard Fourier multiplier argument on $M$, we derive two estimates,
    \begin{align}
    \label{EQL03}
        \|\nabla u\|_{4}\lec \|D^2\psi^{\mathrm{odd}}\|_{L^4(M)}
        \lec \|\omega^{\mathrm{odd}}\|_{L^4(M)}
        \lec \|\omega\|_{\infty},
    \end{align}
    and
    \begin{align}
    \label{EQL04}
        \|\nabla u\|_{\text{BMO}(\Omega)}\lec \|D^2\psi\|_{\text{BMO}(\Omega)}
        \lec \|\omega\|_{\infty}.
    \end{align}
   In order to control $\| \nabla u \|_{W^{1,4}}$, we need to estimate the third order derivatives of~$\psi$. We note that $\partial_2\omega^{\mathrm{odd}}$ might give us a new term around~$\T\times\{0\}$. We consider
    \[
        -\Delta(\partial_1\psi)=\partial_1\omega \quad \text{ in } \Omega \comma \partial_1\omega|_{S\cup\tilde S}=0.
    \]
   An argument analogous to \eqref{EQL03} shows that the $L^4$ norm of each third order derivative $\psi$, except for $\partial_{2}^3\psi$, can be bounded by $\| \omega \|_{W^{1,4}}$. For $\partial_{2}^3\psi$  the PDE in \eqref{EQL01}   gives that
    \[
    \partial_2^3\psi=-\partial_2\omega-\partial_1^2\partial_2\psi.
    \]
    Therefore, we arrive at
    \begin{align}
    \begin{split}
             \label{EQL05}
        \|\nabla u\|_{W^{1,4}}&\lec \|D^2\psi^{\mathrm{odd}}\|_{L^4(M)}+\|D^3\psi^{\mathrm{odd}}\|_{L^4(\Omega)}
        \\& \lec\|\omega^{\mathrm{odd}}\|_{L^4(M)}+\|\partial_1\omega^{\mathrm{odd}}\|_{L^4(M)}
        +\|\partial_2\omega\|_{L^4(\Omega)}\lec
        \|\omega\|_{\infty}+
        \|\nabla\omega\|_4
    \end{split}
    \end{align}
    Applying Lemma~\ref{L_Log01} to $\nabla u$ and using \eqref{EQL04}, and \eqref{EQL05}, we obtain
    \[
    \begin{split}
        \|\nabla u\|_{\infty}&\lec 1+\|\nabla u \|_{BMO(\Omega )}(1+\log^+\| \nabla u \|_{4} ) \\
        &\lec 1+\|\omega\|_{\infty}(1+\log^+(\|\omega\|_{\infty}+
        \|\nabla\omega\|_4))\\
        &\lec 1+\|\omega\|_{\infty}\log(2+\|\omega\|_{\infty}+
        \|\nabla\omega\|_4),
        \end{split}
    \]
    which shows~\eqref{logW14}. The claim \eqref{logHk} follows from \eqref{logW14} using Sobolev embedding.
\end{proof}

\section{Global existence in Sobolev Spaces}
\label{sec_sobolev}

In this section, we prove Theorem~\ref{T01}.\\

We proceed inductively,
starting with
$ \omega^{(0)}\equiv0$.
Given $\omega^{(n)}$, for $n\in\mathbb{N}_0$, we
set 
\eqnb\label{un_def}
u^{(n)} \coloneqq u [\omega^{(n)}]
,
\eqne 
and we define  $\omega^{(n+1)}$ as the solution to the transport equation system 
    \eqnb\label{EQ112}
        \begin{split}
            \partial_t\omega^{(n+1)}+u^{(n)}_1 \partial_1 \omega^{(n+1)}+u^{(n)}_2\partial_2\omega^{(n+1)}&=0, \\
            \omega^{(n+1)}\big|_{t=0}&=\omega_0,\\
            \omega^{(n+1)}\big|_{S}&=\eta(\cdot,t).
        \end{split}
        \eqne
We will treat $\omega_0 $ as $\eta (\cdot ,0)$, extended to~$\Omega$.\smallskip\\

\noindent\texttt{Step 1.} We prove uniform $L^\infty$ estimates.\\\

To this end, given $x\in \Omega$, $t>0$, we set $s_n$ to be the time from which $\onn$ can be represented as following a trajectory from either initial data or boundary condition on $S$, namely,
\[
t_n\coloneqq  \inf  \left\lbrace s \in [0 ,t ] \colon X_n (x,s ;t) \in \Omega \right\rbrace, 
\]
where $X_n(x,s ;t)$ denotes the particle trajectory of $u^{(n)}$ starting at time $t$, i.e., $X_n (x,t;t) \coloneqq x$ and
\[
\p_s X_n (x,s;t) = u^{(n)}(X_n (x,s; t ),s)
\]
for $s \ne t$. In other words, as $s \in [0,t]$ decreases, $X_n (x,s;t)$ is the trajectory moving backwards in time from $x$ (at $s=t$). Thus $X_n (x,t_n;t)$ is either the point on $S$ from which the trajectory passing through $x$ at time $t$ originates (if the infimum is not attained), or $X_n (x,t_n;t)$ is the origin of such trajectory at time $0$ (if $\inf = \min =0=t_n$). In particular we have that
\eqnb\label{om_repr}
\onn (x,t) = \eta (X_n (x,t_n;t ),t_n)
\eqne
for all $x\in \Omega$, $t>0$. Hence,
\eqnb\label{om_Linfty}
\| \onn \|_\infty \leq \sup_{[0,t] } \| \eta (\cdot , t) \|_\infty =: A(t)
\eqne
 for all $n\geq 0$, $t\geq 0$. 
We note in passing that \eqref{om_Linfty} gives that
\[
\left| \int_S u_1^{(n)} \right| \leq \int_0^t \int_\Omega | \on | \lec \int_0^t A = J_2(t),
\]
where we recalled the definition \eqref{EQ15} of~$u[\omega ]$. This implies that
\eqnb\label{un_LinftyS}
\| u^{(n)} \|_{L^\infty (S)} \lec 1+ J_2(t) + \| \p_1  u^{(n)} \|_{L^1 (S) }   \lec 1+ J_2(t) + \| \nabla  \on  \|  
\eqne
for each $t\geq 0$, where we also recalled \eqref{Hkdot} to write $\| \p_1 u^{(n)} \|_{L^2 (S)} \lec \| \nabla u^{(n)}\|_{H^1}\lec \| \on \|_{H^1} \lec J_2 + \| \nabla \on \|$ in the last step.

\begin{Remark}[On the need to control $\| \nabla \omega \|_4$]
{\rm
Having established the $L^\infty$ control in \eqref{om_Linfty}, we now aim to obtain global $H^k$ estimates. However, before this can be achieved, we first need to estimate $\| \nabla \omega \|_4$; let us briefly comment on this issue. For any multiindex $\alpha = (\alpha_1, \alpha_2)$ we have
\eqnb\label{Hkest}
\begin{split}
\frac12 \frac{\d }{\d t } \| D^\alpha \omega \|^2 &= - \underbrace{\int u_1 D^\alpha \omega  \p_1 D^\alpha \omega}_{=0} - \underbrace{ \int u_1 D^\alpha \omega  \p_1 D^\alpha \omega }_{\geq - \int_S | D^\alpha \omega |^2 } - \sum_{\beta < \alpha } {\alpha \choose \beta } \int (D^{\alpha - \beta }u \cdot \nabla D^\beta \omega ) D^\alpha \omega \\
&\leq \| D^\alpha \omega \|_{L^2(S)}^2 +\sum_{\beta < \alpha } {\alpha \choose \beta }  \| D^{\alpha - \beta } u \cdot \nabla D^\beta \omega  \|\, \| D^\alpha \omega \|
.
\end{split} 
\eqne
Here, the boundary term can be handled by reducing $\p_2$ derivatives (using \eqref{EQ10a}) and using the boundary condition in~\eqref{EQ112}. In the case of the $H^2$ estimate (i.e., $|\alpha | =2$) the terms with $|\beta| =1$ can also be handled by using~\eqref{logHk}. The most dangerous term is actually the case $\beta =0$ in the sum, in which case we need to estimate a term of the form 
\[\| D^2 u \|_4 \| \nabla \omega \|_4 \lec \| \nabla \omega \|_4^2.
\]
A natural idea for this would be to use the interpolation $\| \nabla \omega \|_4^2 \lec \| \omega \|_\infty \| D^2 \omega \|$, which would allow us to employ the $L^\infty$ control \eqref{om_Linfty} to close the estimate. However, an interpolation inequality of exactly such form is invalid in our domain, and actually an attempt to generalize it to our setting results with a term involving $\p_2 \omega$ over the \emph{outflow part $\widetilde{S}$ of the boundary}. Indeed, integration by parts gives
\[
\| \p_2 \omega \|_4^4 = \int \p_2 \omega (\p_2 \omega )^3 = - 3\int \omega  (\p_2 \omega )^2 \p_{22} \omega  - \int_{S}  \omega (\p_2 \omega )^3 + \int_{\widetilde{S}}  \omega (\p_2 \omega )^3 .
\]
Here the first term on the right-hand side clearly gives the desired interpolation inequality. While the second term can be handled using the boundary condition for $\omega$, the third one cannot, as we have no information on $\p_2 \omega $ on the outflow part $\widetilde{S}$ of the boundary. This failure of an interpolation inequality is not surprising---the growth of $\| \omega \|_{H^2}$ (or any subcritical norm of $\omega$) is expected to be double exponential (by the BKM condition), while an interpolation would give an exponential growth.\\

Interestingly, the corresponding boundary term in the estimate for $\| \nabla \omega \|_4$ appears with a negative sign (similarly to~\eqref{Hkest} above), see \eqref{W14_est}, and the estimate closes to give a double-exponential growth.
}
\end{Remark}

\noindent\texttt{Step 2.} We show that $\| \nabla \omega\unp \|_4 \leq J_2$ for all $n\geq 0$, $t\geq 0$. \\

To this end  we take the gradient of the PDE in \eqref{EQ112}, multiply by $\nabla \onn|\nabla \onn |^2  $ and integrate to obtain
\eqnb\label{W14_est}
\begin{split}
\frac14 \frac{\d }{\d t} \| \nabla \onn \|_4^4 & = - \int (u_i \p_i \nabla \onn )\cdot  \nabla \onn|\nabla \onn |^2  - \int (\nabla u_i^{(n)} \p_i  \onn  ) \cdot \nabla \onn|\nabla \onn |^2\\
&\leq \int_S | \nabla \onn |^4  + \| \nabla u^{(n)} \|_\infty \| \nabla \onn \|_4^4
.
\end{split}
\eqne

We recall the first order derivative reduction \eqref{EQ08}, which, in the case  of \eqref{EQ112}, gives
\eqnb\label{der_red1_onn}
\p_2 \onn = - \p_t \onn - u_1^{(n)} \p_1 \onn = - \p_t \eta - u_1^{(n)} \p_1 \eta 
\eqne
on $S$, 
so that 
\eqnb\label{temp1}
\| \nabla \onn \|_{L^4 (S)} \lec  J_2  (1+ \| u \|_{L^\infty (S)} ) \lec J_2 (1 + \| \nabla \on \|_4 ),
\eqne
where we used \eqref{un_LinftyS} and recalled, from above Theorem~\ref{T03}, that $J_k (t)$ denotes any constant which may depend on $N_i(s)$ for $i\in \{ 0, \ldots , k \}$ and $s\in [0,t]$.  

Using the log-Lipschitz inequality \eqref{logW14} and \eqref{om_Linfty}, we thus get 
\[
\begin{split}
 \frac{\d }{\d t} \| \nabla \onn \|_4^4 & \lec J_2 \left( \| \nabla \on \|_4^4 +   \| \nabla \onn \|_4^4 \log (2+ \| \nabla \on \|_4^4 ) \right) +J_2,  
\end{split}
\]
which shows that $\| \nabla \omega\unp \|_4$ grows in time at most double-exponentially, and so is bounded by $J_2$ for all times, as required. \\

We note in passing that Step 2 also implies \footnote{This is for convenience only, we could have instead bounded $\| \nabla u^{(n)}\|_\infty $ by $\| \omega \|_\infty$ times a logarithmic factor which could be incorporated directly into the $H^k$ estimate in \eqref{H2_temp} and \eqref{Hk_preest}}
\eqnb\label{un_W1infty}
\| u^{(n) } \|_{W^{1,\infty }} \lec J_2
.
\eqne

\begin{Remark}[The time derivatives]\label{R_time_ders}
{\rm 
Before we continue with the $H^k$ estimates, we note that we need to incorporate the time derivatives. To be precise, we need to  estimate $\| D^\alpha \on \|$ for all $\alpha\in \N_0^3$, not only those with $\alpha_3=0$. This is a consequence of the inflow vorticity, i.e., the first term on the right-hand side of \eqref{EQ18}, i.e., $\int_S (D^\alpha \omega )^2$. Indeed, if $\alpha_2>0$, this term requires us to use derivative reduction \eqref{EQ10a} (or rather its version for the $\omega\unp$ in \eqref{Dalphaonn} below) to ``replace'' $\p_2$ derivatives by the $\p_t$ and $\p_3$ derivatives. This way,  we can use the assumption \eqref{EQ140} on the incoming vorticity $\eta$ (if there are no $\p_2$ derivatives left) or obtain lower order term via a trace estimate (if there are some $\p_2$ derivatives). This will make us, in the following step, to use a double induction to estimate $\| D^\alpha \onn \|_{L^2(S)}$ and then use it to estimate $\| D^\alpha \onn \|$.
}
\end{Remark}

\noindent\texttt{Step 3.} We show that
\begin{eqnarray}
\| D^\alpha \onn \|_{L^2(S)} &\leq& J_k \left( 1+  \sum_{|\alpha'|=k,\, \alpha'_3 = \alpha_3} \| D^{\alpha'} \on \| \right), \label{L2Sa}\\
\| D^\alpha \onn \| &\leq&  J_k \label{L2Sb}
\end{eqnarray}
for all $n\geq 0$, $t\geq 0$, $k\geq 1$ and all $\alpha \in \N_0^3 $ such that $|\alpha |=k$.\\

The case $k=1$ follows from \eqref{EQ140}, \eqref{temp1}, and~\eqref{W14_est}. For $k\geq 2$ we assume that \eqref{L2Sa}--\eqref{L2Sb} hold for $1,\ldots , k-1$, and we prove \eqref{L2Sa} first. We proceed by induction with respect to $\alpha_2$. If $\alpha_2=0$ then $ \| D^\alpha \onn \|_{L^2(S)} \leq J_k$ for all $n\geq 0$, $t\geq 0$, by the assumption~\eqref{EQ140}. If $\alpha_2 >0$, we take $D^{\alpha-e_2}$ of the PDE in  \eqref{EQ112} (as in the derivative reduction \eqref{EQ10a}) to obtain
\eqnb\label{Dalphaonn}
\begin{split}
D^\alpha \onn = -D^{\alpha - e_2+e_3}  \onn  - D^{\alpha - e_2 } (u_1^{(n)} \p_1 \onn ) - \sum_{0<\beta  \leq \alpha - e_2 } {\alpha - e_2 \choose \beta } D^\beta  u_2^{(n)} D^{\alpha - \beta}  \onn   .
\end{split}
\eqne
Note that if all $k-1$ derivatives on the right-hand side of \eqref{Dalphaonn} fall  on $u^{(n)}$, then we can compute the $L^2 (S)$ norm directly by
\[
\begin{split}
\| D^{\alpha-e_2 } u^{(n)} \cdot \nabla  \onn \|_{L^2(S)} &\leq   \| D^{\alpha-e_2 } u_1^{(n)} \p_1 \eta  \|_{L^2(S)} + \| D^{\alpha-e_2 } u_2^{(n)} (\p_t \eta + u_1^{(n)} \p_1 \eta   )   \|_{L^2(S)} \\
&\lec_k  J_k + J_2 \sum_{i=1,2}\left( \| D^{\alpha -e_2+e_i} \on \| + \| D^{\alpha -e_2} \on  \| +\| D^{\alpha -e_2-e_3} \on \|\right)\\
&\lec J_{k}  + J_2  \sum_{\substack{|\alpha'|=k \\ \alpha'_3 = \alpha_3}} \| D^{\alpha'} \on \|  ,
\end{split}
\]
where, in the first step we used the first order derivative reduction \eqref{der_red1_onn} for $\p_2 \onn $, and, in the second step, we used the inductive assumption (with respect to $k$) to absorb the lower-order terms and we  recalled the Biot-Savart inequality \eqref{Hkdot}, as well as recalled  \eqref{EQ50}--\eqref{EQ49} to bound the lowest order term $\| D^{\alpha -e_2 } u \|$ by $\| D^{\alpha - e_2} \on \| + \| D^{\alpha -e_2-e_3} \on \|$ (where the last term arises from $\int_\Omega D^{\alpha-e_2} u_1$ if $\alpha_3>0$). We also used the inductive assumption (with respect to $k$) in the last step.

Moreover, note also that we can bound the $L^2(S)$ norm of all terms on the right-hand side of \eqref{Dalphaonn} with $k$ derivatives on $\onn$ by
\[
\| D^{\alpha - e_2 +e_3 } \onn \|_{L^2(S)} + \| u_1^{(n)} D^{\alpha - e_2 +e_1 } \onn \|_{L^2(S)} \leq J_2 \sum_{\substack{|\alpha'|= k \\  \alpha'_2 \leq \alpha_2-1} } \| D^{\alpha'} \onn \|_{L^2 (S)} \lec J_k,
\]
where we used \eqref{un_W1infty} in the first inequality and the inductive assumption (with respect to $\alpha_2$) in the second inequality.

Applying these two observations, as well as the inductive assumption (with respect to both $\alpha_2$ and $k$) to handle the first term on the right-hand side of \eqref{Dalphaonn}, we can now compute the $L^2(S)$ norm of \eqref{Dalphaonn},
\[\begin{split}
\| D^\alpha \onn \|_{L^2 (S)} &\lec_k J_k +J_2  \sum_{|\alpha'| = |\alpha |} \| D^{\alpha'} \on \| + J_2 \sum_{\substack{0<\beta < \alpha-e_2 \\ | \beta | = |\alpha -2|}} \| D^\beta u^{(n)} \|_\infty \| D^{\alpha - \beta } \onn \|_{L^2(S)} \\
&+J_2 \sum_{\substack{0<\beta < \alpha-e_2 \\ |\beta |\leq |\alpha |-3}}\,\,\, \underbrace{\| D^\beta u^{(n)} \|_\infty }_{\lec \| D^\beta \on \|_{H^2} \leq J_{k-1} }\underbrace{\| D^{\alpha - \beta } \onn \|_{L^2(S)}}_{\leq J_{k-1} }\\
&\leq J_k \left( 1+ \sum_{|\alpha'|=k, \alpha'_3 = \alpha_3} \| D^{\alpha'} \on \| \right)  + J_2 \sum_{\substack{0<\beta < \alpha-e_2 \\ | \beta | = |\alpha -2|}} \| D^\beta u^{(n)} \|_{H^2}
\\&
%\leq J_k \left( 1+ \sum_{|\alpha'|=k, \alpha'_3 = \alpha_3} \| D^{\alpha'} \on \| \right) ,
\leq J_k + J_k \sum_{|\alpha'|=k, \alpha'_3 = \alpha_3} \| D^{\alpha'} \on \|
,
\end{split}
\]
as required, where, in the second inequality, we noted that $\| D^{\alpha -\beta } \onn \|_{L^2(S)} \leq J_2$ if $k\geq 3$ (by the inductive assumption with respect to $k$), and that this term is excluded from the summation if $k=2$. Moreover, in the last inequality we used the inductive assumption (with respect to $k$) to absorb the lower order terms resulting from \eqref{Hkdot}--\eqref{EQ49} into~$J_k$.\\

Having established \eqref{L2Sa} for $|\alpha |=k$, we now show~\eqref{L2Sb}. Let $k=2$. We proceed by induction with respect to~$\alpha_3$. If $\alpha_3=0$ we use  \eqref{W14_est} and \eqref{L2Sa} to continue  \eqref{Hkest} as follows
\eqnb\label{H2_temp}
\begin{split}
\frac{\d }{\d t } \| D^\alpha \onn \|^2 &\lec  \| D^\alpha \onn \|_{L^2 (S)}^2 +  \sum_{\beta < \alpha } {\alpha \choose \beta }  \| D^{\alpha - \beta } u^{(n)} \cdot \nabla D^\beta \onn  \|\, \|  \onn \|_{H^2}\\
& \lec  J_2 (1+ \| \on \|_{H^2}^2 )+\left( \| D^\alpha u^{(n)} \|_4 \| \nabla \onn \|_4+ \| \nabla u^{(n)} \|_\infty \| \onn \|_{H^2} \right)  \|  \onn \|_{H^2}\\
& \lec J_2 \left( 1+\|  \on \|_{H^2}+\|  \onn \|_{H^2} \right)^2 ,
\end{split} 
\eqne
where we used \eqref{un_W1infty} and \eqref{W14_est} (again) in the last step. By induction in $n$ this shows that 
\eqnb\label{H2_est}
\| \on \|_{H^2} \lec J_2 \qquad \text{ for all }t\geq 0, n\geq 0,
\eqne
which completes the proof of the case $\alpha_3=0$. If $\alpha_3>0$, the claim follows from the PDE in \eqref{EQ112}, the inductive assumption (with respect to $\alpha_3$), and induction in~$n$. 

It remains to prove \eqref{L2Sb} in the case $k\geq 3$. Similarly to the case $k=2$, we proceed by induction in~$\alpha_3$. If $\alpha_3=0$, then
\eqnb\label{Hk_preest}
\begin{split}
\frac{\d }{\d t } \| D^\alpha \onn \|^2 &\lec  J_k \left( 1 + \| \on \|_{H^k}^2\right)  +  \sum_{\beta < \alpha } {\alpha \choose \beta }  \| D^{\alpha - \beta } u^{(n)} \cdot \nabla D^\beta \onn  \|\, \|  \onn \|_{H^2}
,
\end{split} 
\eqne
where we used \eqref{L2Sa} to obtain the first term on the right-hand side. As for the terms inside the sum, we have 
\[
\| D^\alpha u^{(n)} \cdot \nabla \onn \| \leq \|  u^{(n)} \|_{H^k} \| \nabla \onn \|_\infty \lec J_{k-1} \| \onn \|_{H^3} \lec J_{k-1} \| \onn \|_{H^k} 
\]
for $\beta=0$, 
\[
\| D^{\alpha -\beta } u^{(n)} \cdot \nabla D^\beta  \onn \| \leq \|    \nabla u^{(n)} \|_{W^{k-2,4}} \| \nabla D^\beta \onn \|_4  \lec_k \| \on \|_{H^{k-1} } \| \onn \|_{H^3}\leq J_{k-1} \| \onn \|_{H^k }  
\]
for $|\beta |=1$, and
 \[
\| D^{\alpha-\beta } u^{(n)} \cdot \nabla D^\beta \onn \| \leq \| D^{\alpha - \beta } u^{(n)} \|_{\infty } \| \nabla D^\beta  \onn \| \lec \| \nabla u^{(n)} \|_{H^{k-|\beta | +1}} \| \onn \|_{H^k}\leq J_{k-1} \| \onn \|_{H^k }
\]
for $0<\beta <\alpha$, $|\beta |\geq 2$. Applying these in \eqref{Hk_preest} we get 
 \[
\frac{\d }{\d t } \| D^\alpha \onn \|^2 \lec  J_k \left( 1+ \| \on \|_{H^k} + \| \onn \|_{H^k} \right)^2
\]
for all $|\alpha |=k$, which proves \eqref{L2Sb} in the case $k\geq 3$ and $\alpha_3=0$, due to Gronwall's inequality. Similarly to the $k=2$ case, the PDE in \eqref{EQ112}, the inductive assumption (with respect to $\alpha_3$) and induction in $n$, gives  \eqref{L2Sb} for $k\geq 3$ and $\alpha_3>0$, as required. \smallskip\\

\noindent\texttt{Step 4.} We use the global bounds \eqref{L2Sa}--\eqref{L2Sb} to construct the global-in-time solution claimed by Theorem~\ref{T01}. \\

We set
\[
\tom^{(n)} \coloneqq \omega^{(n)}-\omega^{(n-1)},\qquad  \tu\un \coloneqq  u\un-u\unm = u[ \tom\un ].
\]
Note that we have 
\begin{align}
\begin{split}
    &\tom\un|_{S}=0,
    \\&
    \tom\un(0)=0,
    \\&
    \int_{\Omega}\tu_1\un\,\d\sigma=\int_0^t\int_{\Omega}\tom\un,
\end{split}
    \label{EQ127}
\end{align}
and $\tom\unp$ satisfies
\begin{align}
           \partial_t\tom\unp + u\un\cdot \nabla \tom\unp +\tu\unp \cdot \nabla \omega\unp=0.
    \label{EQ129}
\end{align}
In particular, this gives that
\eqnb\label{tom1}
\p_2 \tom\unp = \underbrace{- \p_t \tom\unp - u_1\un \p_1 \tom\unp}_{=0} - \tu\unp \cdot \nabla \omega\unp
\eqne
on~$S$. More generally, for each $\alpha $ with $\alpha_2>0$,
\eqnb\label{tom2}
D^\alpha  \tom\unp = -D^{\alpha-e_2 + e_3} \tom\unp - D^{\alpha-e_2} (u\un \cdot \nabla  \tom\unp )+ u_2\un D^\alpha \tom  -D^{\alpha-e_2} ( \tu\unp \cdot \nabla \omega\unp )
.
\eqne
Thus, using the uniform bounds \eqref{L2Sa}--\eqref{L2Sb}, we get
\eqnb\label{tom3}
\| D^\alpha \tom\unp \|_{L^2(S)} \leq J_{k+1} \sup_{|\beta | \leq k } \| D^\beta \tom\unp \| 
\eqne
for all $|\alpha | =k$, $k\geq 2$, $n \geq 0$, $t\geq 0$. Indeed, \eqref{tom3} follows by induction in $k$, where, for each $k\geq 2$, (similarly to the proof of~\eqref{L2Sa}) the claim is proved by induction with respect to $\alpha_2 \in \{ 0, \ldots , k \}$. 

Furthermore, we have
\eqnb\label{tom4}
\frac{\d }{\d t}
\| D^\alpha \tom\unp \| \leq J_{k+1} \left(  \sup_{|\beta | \leq k } \| D^\beta \tom\unp \|
+ \sup_{|\beta | \leq k } \| D^\beta \tom\un \|  \right)
\eqne
for all $|\alpha | =k$, $k\geq 2$, $n \geq 0$, $t\geq 0$. Indeed, \eqref{tom4} follows in a similar way as \eqref{L2Sb}: We use the induction with respect to $k\geq 2$, where, for each $k\geq 2$, we use the induction with respect to~$\alpha_3$. Then \eqref{tom3} lets us control $ \| D^\alpha \tom\unp \|_{L^2(S)}$, and the bounds \eqref{L2Sa}--\eqref{L2Sb} let us control the terms arising from $u\un$ and $\omega\unp$ appearing in~\eqref{EQ129}. We omit the details, but we emphasize that the subindex $k+1$ in \eqref{tom4} arises from the derivative loss from the term $\tu\unp\cdot \nabla \omega\unp$ in~\eqref{tom1}.

We now fix $k\geq 2$ and set 
\[\tilde{y}^{(n)}\coloneqq \sup_{|\alpha |\leq k} \|D^\alpha \tom^{(n)}\|,\]
so that \eqref{tom4} gives that 
\[
 \tilde{y}^{(n+1)} (t_1) - \tilde{y}^{(n+1)} (t_0) 
 \leq \int_{t_0}^{t_1} J_{k+1} \left( \tilde{y}^{(n+1)} + \tilde{y}^{(n)} \right)
\]
for all $t_1>t_0\geq 0$, $n\geq 1$, which is an integral representation of 
\begin{align}
    \frac{\d}{\d t}\left( \ee^{-\int_0^t J_{k+1} }\tilde{y}\unp  \right)\leq J_{k+1} \ee^{-\int_0^t J_{k+1}}\tilde{y}\un ,
    \label{EQ131}
\end{align}
from where, integrating $n$ times in $t$ we obtain that 
\begin{align}
  \tilde{y}^{(n+1)}(t)\leq \ee^{\int_0^t J_{k+1}}
\int_0^t J_{k+1}(s)\ee^{-\int_0^s J_{k+1}}\tilde{y}^{(1)}(s)\frac{\left(\int_s^t J_{k+1}\right)^{n-1}}{(n-1)!}\,\text{d}s .
    \label{EQ132}
\end{align}
Thus, since the right-hand side is summable in $n$ for each $t\geq 0$, we see that 
\[
D^\alpha \tom^{(1)} + \sum_{n\geq 1} D^\alpha \tom\un \qquad \text{ converges  in } C([0,t];L^2)
\]
for every $t\geq 0$ and $\alpha\in \N_0^3 $ such that $|\alpha |\leq k$. Since  $C([0,t];L^2)$ is a Banach space, this implies that
\[
D^\alpha \omega^{n} \qquad \text{ converges to  } \omega \text{ in } C([0,t];L^2)
\]
for all such $t,\alpha$, for some $\omega \in C([0,\infty );L^2)$, such that $D^\alpha \omega \in C([0,\infty );L^2)$ for all $\alpha \in \N_0^3$ such that $|\alpha |\leq k$. Taking the limit $n\to \infty $ in \eqref{EQ112} gives the Euler equation for $\omega$ with velocity field $u[\omega ]$, as required.

In order to verify the velocity formulation \eqref{EQ01}, we note that $   F\coloneqq \partial_t u+u\cdot\nabla u$ satisfies
$\curl F=0$ and 
\begin{align}
\begin{split}
    \int_\Omega  F_1 
    &
    =\int_\Omega (\partial_t u_1+u_1\partial_1u_1+u_2\partial_2u_1)
    \\&
    = \frac{\d }{\d t} \int_\Omega u_1+\underbrace{\int_\Omega\frac{1}{2}\partial_1 u^2_1}_{=0}+\int_{\tilde{S}}u_1-\int_{S}u_1-\int_\Omega\underbrace{\partial_2u_2}_{=-\p_1 u_1}u_1
    \\&
    = \frac{\d }{\d t} \int_\Omega u_1+\int_{\tilde{S}}u_1-\int_{S}u_1 +\underbrace{ \int_\Omega \p_1 u_1^2}_{=0}
    = \frac{\d }{\d t}\int_\Omega u_1+\int_\Omega \partial_2u_1 -\underbrace{\int_\Omega \partial_1u_2}_{=0}
    \\&
    = \frac{\d }{\d t}\left(\int_\Omega u_1-\int_0^t\int_\Omega\omega\right)
    =0
    \end{split}
    \label{EQ136}
\end{align}
for each $t>0$. This lets us use the De Rham Lemma~\ref{L05} to deduce that $F=\nabla p$ for some pressure function $p$, concluding the proof.

\section{Induction for boundary Gevrey coefficients $a_m$}
\label{sec_bdry_coeffs}

Here we prove~\eqref{bdry_coefs}. Namely, we assume that 
\eqnb\label{bdry_c_as}
\epsilon_2 (t) \leq \min \left( \frac{\epsilon_3}{4K(t)}, \frac{\epsilon_1}{4}\right),\qquad 
\tau (t) \leq \frac{1}{C_s  K^3  (t)},
\eqne
where $C_s$ is from \eqref{Cs_def}, and we show that, for each $m\geq 3$, $b_m \leq K (t)$ implies that $a_{m+1} \leq K (t)$. For brevity, we will simply write $K\equiv K(t)$. \\

To this end, we first note that the vorticity equation \eqref{EQ07} implies that 
$\eta_t + u_1\partial_{1}\eta + u_2\partial_{2}\omega=0$ on $S$,
which gives the expression
  \begin{equation}
   \partial_{2}\omega
   = - \eta_t - u_1\partial_{1}\eta
   \onon{S}
   .
   \label{EQ08}
  \end{equation}
To obtain an identity for higher order derivatives
of $\omega$ on $S$, we proceed as follows.
Denote by $\alpha=(\alpha_1,\alpha_2,\alpha_3)\in\mathbb{N}_{0}^3{}$
an arbitrary multiindex. With the notation
$\partial^{\alpha}=\partial_{1}^{\alpha_1}\partial_{2}^{\alpha_2}\partial_{t}^{\alpha_3}$,
we have
  \begin{equation}
   \partial^{\alpha+e_3}\omega
   + \partial^{\alpha} (u_1\partial_{1}\omega)
   + \partial^{\alpha}(u_2\partial_2 \omega)=0
   \onon{\Omega},
   \label{EQ09}
  \end{equation}
which leads to
\begin{equation}
   u_2\partial^{\alpha+e_2}\omega
   =
   - \partial^{\alpha+e_3}\omega
   - \partial^{\alpha}(u_1\partial_{1}\omega)
   -
   \bigl(
    \partial^{\alpha}(u_2\partial_2\omega)
    - u_2\partial^{\alpha+e_2}\omega
   \bigr)
   \onon{\Omega}
   .
   \label{EQ10b}
  \end{equation}
Restricting to $S$, we obtain
\begin{equation}
   \partial^{\alpha+e_2}\omega
   =
   - \partial^{\alpha+e_3}\omega
   - \partial^{\alpha}(u_1\partial_{1}\omega)
   -
   \bigl(
    \partial^{\alpha}(u_2\partial_2\omega)
    - u_2\partial^{\alpha+e_2}\omega
   \bigr)
   \onon{S}
   ,
   \label{EQ10a}
  \end{equation}
which is used in the sequel.\\

In order to show that $a_{m+1}\leq K$, it suffices to verify the claim for each $a_{\alpha'}$, where $|\alpha' |= m+1$. We use induction with respect to~$\alpha'_2$. The base case $\alpha'_2=0$ follows from \eqref{a_K_bound}, while for $\alpha'_2\geq 1$ we use the derivative reduction formula \eqref{EQ10a} to write $\alpha'=\alpha +e_2$, so that 
\[
    \begin{split}
    a_{\alpha+e_2}&
    =\frac{\epsilon^{\alpha+e_2}\tau^{|\alpha|-1}}{(|\alpha|+1)!^s}\|D^{\alpha+e_2}\omega\|_{L^2(S)}
    \\&
    =
    \frac{\epsilon^{\alpha+e_2}\tau^{|\alpha|-1}}{(|\alpha|+1)!^s}
    \bigl\|
    D^{\alpha+e_3}\omega
    +D^{\alpha}(u_1\partial_{1}\omega)
    +
   \bigl(
    D^{\alpha}(u_2\partial_2\omega)
    - u_2 D^{\alpha+e_2}\omega
   \bigr)\bigr\|_{L^2(S)}
,\\
&\leq \frac{\epsilon_2}{\epsilon_3} \underbrace{\frac{\epsilon^{\alpha+e_3}\tau^{|\alpha|-1}}{(|\alpha|+1)!^s}
    \|D^{\alpha}\partial_t\omega\|_{L^2(S)}}_{ =a_{\alpha + e_3}}
+\frac{\epsilon_2}{\epsilon_1}\underbrace{\|u\|_{\lis} }_{\leq K} \underbrace{ \frac{\epsilon^{\alpha+e_1}\tau^{|\alpha|-1}}{(|\alpha|+1)!^s}
    \|D^{\alpha+e_1}\omega\|_{\ls}}_{=a_{\alpha +e_1}} \\
    &\,\,\,    +\frac{\epsilon^{\alpha+e_2}\tau^{|\alpha|-1}}{(|\alpha|+1)!^s}
    \sum_{i=1,2}
    \sum_{0< \beta<\alpha}\binom{\alpha}{\beta}\|D^\beta u\|_{\lis}
    \|D^{\alpha-\beta+e_i}\omega\|_{\ls}
    \\&\,\,\, 
       +\frac{\epsilon^{\alpha+e_2}\tau^{|\alpha|-1}}{(|\alpha|+1)!^s}\|D^\alpha u\|_{\ls}
    \| \nabla \omega\|_{\lis}
    ,
    \end{split}
     \]
where, in the last step, we recalled \eqref{EQ171a} and noted that the second and third terms inside the preceding norm have a similar structure and can thus be treated similarly.  We note that the  second and the last terms on the right-hand side correspond to the cases ($i=1$, $\beta = 0$) and $\beta =\alpha$, respectively.
Using \eqref{bdry_c_as} and the Sobolev inequality $\| f \|_{L^\infty (S)} \lec \| f \|_{H^3} \lec \| f \| + \|\nabla^2 f \|$, to get
\[
\begin{split}
    a_{\alpha+e_2}
    &\leq
    \frac{1}{4}a_{\alpha+e_3} + \frac{1}{4}  a_{\alpha+e_1}    +C \frac{\epsilon^{\alpha+e_2}\tau^{|\alpha|-1}}{(|\alpha|+1)!^s}
    \sum_{i=1,2}
    \sum_{0<\beta<\alpha} \binom{\alpha}{\beta}
    \|D^\beta u\|
    \|D^{\alpha-\beta+e_i}\omega\|_{\ls}
    \\&\indeq
    +C \frac{\epsilon^{\alpha+e_2}\tau^{|\alpha|-1}}{(|\alpha|+1)!^s}
    \sum_{i=1,2}
    \sum_{0<\beta<\alpha} \binom{\alpha}{\beta}
    \|D^\beta \nabla^2 u\|
    \|D^{\alpha-\beta+e_i}\omega\|_{\ls}
    \\&\indeq
    +C \frac{\epsilon^{\alpha+e_2}\tau^{|\alpha|-1}}{(|\alpha|+1)!^s}
    \left(\|D^\alpha u\|+\|D^\alpha \nabla u\|\right)
    \|\omega\|_{H^3}
    \\&
    =: \frac{1}{4}a_{\alpha+e_3} + \frac{1}{4}  a_{\alpha+e_1}   +I_1+I_2+I_3.
       \end{split}
\]
Next, we combine the  velocity estimates \eqref{EQ47}--\eqref{EQ67} to bound the contributions of
$\|D^\beta u\|_{\lo}$ and $\|D^\beta\nabla^2u\|_{\lo}$ separately.
For $I_1$, we have
\begin{align}
\begin{split}
    I_1
    &\lec
      \tau^2
      \sum_{i=1,2}
      \sum_{0<\beta<\alpha,|\beta|\geq 2}
      \underbrace{\frac{\epsilon^\beta\tau^{|\beta|-2}}{|\beta|!^s}\|D^\beta u\|}_{\leq b_{m-1}\leq K }
      \frac{\epsilon_2}{\epsilon_i}
    \underbrace{\frac{\epsilon^{\alpha-\beta+e_i}\tau^{|\alpha|-|\beta|-1}}
    {(|\alpha|-|\beta|+1)!^s}
    \|D^{\alpha-\beta+e_i}\omega\|_{\ls}}_{
    =a_{\alpha - \beta +e_i } \leq K}
    \binom{\alpha}{\beta}\binom{|\alpha|+1}{|\beta|}^{-s}
    \\&\indeq
    +
    \tau
    \sum_{i=1,2}\sum_{|\beta|=1}
    \frac{\epsilon^\beta}{|\beta|!^s}
  \underbrace{  \|D^\beta u\|}_{ \leq K }
    \frac{\epsilon_2}{\epsilon_i}
   \underbrace{ \frac{\epsilon^{\alpha-\beta+e_i}\tau^{|\alpha|-|\beta|-1}}{(|\alpha|-|\beta|+1)!^s}
    \|D^{\alpha-\beta+e_i}\omega\|_{\ls}}_{ \leq K}
    \binom{\alpha}{\beta}\binom{|\alpha|+1}{|\beta|}^{-s}
    \\&
    \lec
   K^2  \tau  \sum_{ \beta < \alpha , |\beta|=1}\binom{\alpha}{\beta}\binom{|\alpha|+1}{|\beta|}^{-s}
    ,
    \label{EQ76}
    \end{split}
\end{align}
where we used \eqref{EQ67} and   \eqref{EQ171a}  in the second step.
Proceeding further, we employ the combinatorial identity \eqref{EQ77}, together with the inequalities~\eqref{EQ38}, we get
\begin{align}
    I_1
      \lec_s \tau K^2 
      .    
    \label{EQ78}
\end{align}

Note that, as opposed to $I_1$, the derivatives of factors
in $I_2$ always contain at least one spatial derivative, and so \eqref{EQ70} is not needed.
The difference between $I_1$ and $I_2$ is that in $I_1$ the derivatives $D^\beta$ may consist entirely of time derivatives, in which case we need to use~\eqref{EQ70}. In contrast, in $I_2$, the term $D^\beta\nabla^2u$ necessarily contains spatial derivatives, and thus \eqref{EQ67} suffices.
Proceeding similarly as for $I_1$, we estimate $I_2$ as 
  \begin{align}
    \begin{split}
    I_2&= C \frac{\epsilon^{\alpha+e_2}\tau^{|\alpha|-1}}{(|\alpha|+1)!^s}
    \sum_{i=1,2}
    \sum_{0<\beta<\alpha} \binom{\alpha}{\beta}
    \|D^\beta \nabla^2 u\|_{\lo}
    \|D^{\alpha-\beta+e_i}\omega\|_{\ls}
    \\&\lec 
    \sum_{0<\beta<\alpha}\sum_{i,j,l=1,2}\underbrace{
    \frac{\epsilon^{\beta+e_j+e_l}\tau^{(|\beta|+2)-2}}{(|\beta|+2)!^s}\|D^\beta \partial_j \partial_l u\|}_{\leq \tau |\alpha|^{-2}b_m \leq \tau |\alpha |^{-s} K} \underbrace{
    \frac{\epsilon^{\alpha-\beta+e_i}\tau^{|\alpha-\beta|-1}}{(|\alpha-\beta|+1)!^s}
    \|D^{\alpha-\beta+e_i}\omega\|_{\ls}}_{ = a_{\alpha -\beta +e_i}\leq K }
    \\&\indeq\times
    \binom{\alpha}{\beta}
    \binom{|\alpha|}{|\beta|}^{-s}
    \frac{(|\beta|+1)^{s}(|\alpha-\beta|+1)^{s}}{(|\alpha|+1)^s}
    (|\beta|+2)^s
    \frac{\epsilon^{e_2}}{\epsilon^{e_i+e_j+e_l}}
    \\&\lec 
     |\alpha|^{-s}\tau K^2 
     \sum_{0<|\beta|<|\alpha|}
    \binom{|\alpha|}{|\beta|}^{1-s}
    \frac{(|\beta|+1)^{s}(|\alpha-\beta|+1)^{s}}{(|\alpha|+1)^s}
    (|\beta|+2)^s
    \underbrace{\sum_{i,j,l=1,2}\frac{\epsilon^{e_2}}{\epsilon^{e_i+e_j+e_l}}}_{\lec K^2},
    \label{EQ81}
    \end{split}
  \end{align}
  where we also used the combinatorial inequality \eqref{EQ77} and \eqref{EQ67} in the last step. 
Using the binomial inequality \eqref{EQ30}, we thus get 
\begin{align}
    \begin{split}
    I_2&   
    \leq
  C\tau    K^4
    \underbrace{
    \sum_{1\leq k\leq|\alpha|-1}\binom{|\alpha|}{|k|}^{1-s}
    \frac{(k+1)^s(|\alpha|-k+1)^s}{(|\alpha|+1)^s}
    }_{\lec_s 1}
    \underbrace{ \frac{(k+2)^s}{|\alpha|^s}}_{ \lec_s 1}
    \\&\leq
     C_s \tau K^4 .
        \label{EQ83}
    \end{split}
\end{align}

Finally, we bound $I_3$ as
\begin{align}
\begin{split}
    I_3
    &=
    C\frac{\epsilon^{\alpha+e_2}\tau^{|\alpha|-1}}{(|\alpha|+1)!^s}
    \left(\|D^\alpha u\|+\|D^\alpha \nabla u\|\right)
  \underbrace{  \|\omega\|_{H^3}}_{\leq K}
    \\&\lec 
    K \left(
    \frac{\epsilon_2\tau}{(|\alpha|+1)^s}
    \underbrace{\frac{\epsilon^\alpha\tau^{|\alpha|-2}}{(|\alpha|)!^s}
    \|D^\alpha u\|}_{\leq b_{m} \leq K }
    +\sum_{i=1,2}\frac{\epsilon_2}{\epsilon_i}
   \underbrace{ \frac{\epsilon^{\alpha+e_i}\tau^{|\alpha|-1}}{(|\alpha|+1)!^s}\|D^{\alpha+e_i}u\|}_{\leq \tau |\alpha |^{-s} b_{m} }
    \right)
    \\&
    \lec K^2  \tau|\alpha|^{-s}
    ,
    \label{EQ85}
    \end{split}
\end{align}
where we used \eqref{EQ67} in the last step. Adding \eqref{EQ78}, \eqref{EQ83}, and \eqref{EQ85}, we thus obtain
\eqnb\label{Cs_def}
a_{\alpha + e_2} \leq \frac{1}{4}a_{\alpha+e_3} + \frac{1}{4}  a_{\alpha+e_1}   + C_s \tau K^4 \leq K,
\eqne
as required, where we used the inductive assumption that $a_{\alpha+e_1},a_{\alpha+e_3} \leq K$ and the smallness assumption \eqref{bdry_c_as} on $\tau$ in the last inequality.

\section{The non-linear term}
\label{sec_nonlin}

Here we show that
\eqnb\label{nonlin_to_show}
        \max_{|\alpha|=m+1} \NL(\alpha) 
	\lec
         m  K^4  b_{m+1}^2 
    \eqne
for all $m\geq 0$ and $\alpha \in \N_0^3$ such that $|\alpha |=m+1$, provided $b_m \leq K$. As in the previous section, we use the notation $K\equiv K(t)$, for brevity.

We first apply H\"older’s inequality to deduce that
\begin{align}
    \begin{split}
        \NL(\alpha)
        &\leq
    \frac{\epsilon^{2\alpha}\tau^{2|\alpha|-4}}{|\alpha|!^{2s}}\sum_{\substack{0<\beta<\alpha\\2\leq|\beta|\leq|\alpha|-2}}\binom{\alpha}{\beta}
    \|D^\beta u\|_{\infty}\|D^{\alpha-\beta}\nabla\omega\|\|D^\alpha\omega\|
    \\&\indeq
    +\frac{\epsilon^{2\alpha}\tau^{2|\alpha|-4}}{|\alpha|!^{2s}}
    \|D^\alpha u\|_{4}\|\nabla\omega\|_{4}\|D^\alpha\omega\|
    \\&\indeq
    +\frac{\epsilon^{2\alpha}\tau^{2|\alpha|-4}}{|\alpha|!^{2s}}
    \sum_{|\beta|=1}\binom{\alpha}{\beta}
    \|D^\beta u\|_{\infty}\|D^{\alpha-\beta}\nabla\omega\|\|D^\alpha\omega\|
    \\&\indeq
    +\frac{\epsilon^{2\alpha}\tau^{2|\alpha|-4}}{|\alpha|!^{2s}}
    \sum_{|\beta|=|\alpha|-1}\binom{\alpha}{\beta}
    \|D^\beta u\|_{4}\|D^{\alpha-\beta}\nabla\omega\|_{4}\|D^\alpha\omega\|.
    \end{split}
\end{align}
Using Sobolev inequality, we obtain
\begin{align}
\begin{split}
    \NL(\alpha)&
    \lec \frac{\epsilon^{2\alpha}\tau^{2|\alpha|-4}}{|\alpha|!^{2s}}\sum_{\substack{0<\beta<\alpha\\2\leq|\beta|\leq|\alpha|-2}}
    \sum_{i=1,2}\binom{\alpha}{\beta}
    \left(\|D^\beta u\|+\|D^\beta\nabla^2u\|\right)
      \|D^{\alpha-\beta}\partial_i\omega\|\|D^\alpha\omega\|
    \\&\indeq+ \frac{\epsilon^{2\alpha}\tau^{2|\alpha|-4}}{|\alpha|!^{2s}}\sum_{i,j,k=1,2}
    \left(\|D^\alpha u\|+\|D^\alpha \partial_iu\|\right)
    \left(\|\partial_j\omega\|+\|\partial_j\partial_k\omega\|\right)\|D^\alpha\omega\|
    \\&\indeq
    +\frac{\epsilon^{2\alpha}\tau^{2|\alpha|-4}}{|\alpha|!^{2s}}
    \sum_{|\beta|=1}\binom{\alpha}{\beta}
    \|D^\beta u\|_{\infty}\|D^{\alpha-\beta}\nabla\omega\|\|D^\alpha\omega\|
    \\&\indeq
    +\frac{\epsilon^{2\alpha}\tau^{2|\alpha|-4}}{|\alpha|!^{2s}}
    \sum_{|\beta|=|\alpha|-1}\binom{\alpha}{\beta}
    \|D^\beta u\|_{4}\|D^{\alpha-\beta}\nabla\omega\|_{4}\|D^\alpha\omega\|
    \\&=:(\NL_{1,1} + \NL_{1,2} )+ \NL_2+ \NL_3+ \NL_4.
    \label{EQ92}
    \end{split}
\end{align}
We next proceed to estimate each of the nonlinear terms individually. First,
\begin{align}
    \begin{split}
    \NL_{1,1}&=
    \frac{\epsilon^{2\alpha}\tau^{2|\alpha|-4}}{|\alpha|!^{2s}}\sum_{\substack{0<\beta<\alpha\\2\leq|\beta|\leq|\alpha|-2}}
    \sum_{i=1,2}\binom{\alpha}{\beta}
    \|D^\beta u\|\|D^{\alpha-\beta}\partial_i\omega\|\|D^\alpha\omega\|
    \\&=
    \frac{\epsilon^\alpha\tau^{|\alpha|-2}}{|\alpha|!^s}\|D^\alpha\omega\|
    \sum_{\substack{0<\beta<\alpha\\2\leq|\beta|\leq|\alpha|-2}}
    \frac{\epsilon^{\beta}\tau^{|\beta|-2}}{|\beta|!^s}\|D^\beta u\|
    \\&\indeq    \times
    \sum_{i=1,2}\frac{\epsilon^{\alpha-\beta+e_i}\tau^{|\alpha-\beta|-1}}{|\alpha-\beta+1|!^s}
    \|D^{\alpha-\beta}\partial_i\omega\|
    \frac{\tau}{\epsilon_i}\binom{\alpha}{\beta}\binom{|\alpha|}{|\beta|}^{-s}|\alpha-\beta+1|^s
    \\&=
    b_{\alpha}\sum_{\substack{0<\beta<\alpha\\2\leq|\beta|\leq|\alpha|-2}}\frac{\epsilon^{\beta}\tau^{|\beta|-2}}{|\beta|!^s}\|D^\beta u\|
    \sum_{i=1,2}b_{\alpha-\beta+e_i}\frac{\tau}{\epsilon_i}\binom{\alpha}{\beta}\binom{|\alpha|}{|\beta|}^{-s}|\alpha-\beta+1|^s
    .
    \label{EQ93}
    \end{split}
\end{align}
Employing the estimates~\eqref{EQ29},  \eqref{EQ67}, and \eqref{EQ70} we derive
\begin{align}
    \begin{split}
        \NL_{1,1}
        &\lec K b_{m+1}^2  \sum_{0<\beta<\alpha} \sum_{i=1,2}   \frac{\tau}{\epsilon_i} \binom{\alpha}{\beta} \binom{|\alpha|}{|\beta|}^{-s}|\alpha-\beta+1|^s
        \\&
        \lec K^2 b_{m+1}^2 \sum_{i=1,2}\frac{\tau}{\epsilon_i}|\alpha|
        \sum_{k=1}^{|\alpha|-1}\binom{|\alpha|}{k}^{1-s}\frac{(|\alpha|-k+1)^s}{|\alpha|}
        \\&
        \lec K^3|\alpha|\tau b_{m+1}^2,
        \label{EQ94a}
    \end{split}
\end{align}
where we used the definition \eqref{eps_rel} of $\epsilon_2$ in the last step. The distinction between $\NL_{1,1}$ and $\NL_{1,2}$ is that $\|D^\beta \nabla^2u\|$ contains spatial derivatives, allowing us to use the stronger estimate~\eqref{EQ67}. Thus,
\begin{align}
    \begin{split}
    \NL_{1,2}
    &=\frac{\epsilon^{2\alpha}\tau^{2|\alpha|-4}}{|\alpha|!^{2s}}\sum_{\substack{0<\beta<\alpha\\2\leq|\beta|\leq|\alpha|-2}}\sum_{i=1,2}\binom{\alpha}{\beta}
    \|D^\beta \nabla^2u\|\,\|D^{\alpha-\beta}\partial_i\omega\|\,\|D^\alpha\omega\|
    \\&
    =\underbrace{\frac{\epsilon^\alpha\tau^{|\alpha|-2}}{|\alpha|!^s}\|D^\alpha\omega\|}_
    {= b_{\alpha}\leq b_{m+1} }\sum_{i,j,k=1,2}
    \sum_{\substack{0<\beta<\alpha\\2\leq|\beta|\leq|\alpha|-2}}\underbrace{\frac{\epsilon^{\beta+e_j+e_k}\tau^{|\beta|+2-2}}{(|\beta|+2)!^s}\|D^\beta\partial_j \partial_ku\|}_{ \leq \tau |\alpha |^{-s} b_m \leq \tau |\alpha  |^{-s} K}
    \\&\indeq
    \times\underbrace{ \frac{\epsilon^{\alpha-\beta+e_i}\tau^{|\alpha-\beta|-1}}{\left(|\alpha-\beta|+1\right)!^s}\|D^{\alpha-\beta}\partial_i\omega\|}_{= b_{\alpha - \beta +e_i}\leq b_{m+1}} \binom{\alpha}{\beta}\binom{|\alpha|}{|\beta|}^{-s}\frac{(|\beta|+1)^s(|\beta|+2)^s(|\alpha-\beta|+1)^s}
    {\tau \,\epsilon_i \epsilon_j\epsilon_k  }
    \\&
    \leq  b_{m+1}^2 K 
    \sum_{i,j,k=1,2}\sum_{\substack{0<\beta<\alpha\\2\leq|\beta|\leq|\alpha|-2}} \binom{\alpha}{\beta} \binom{|\alpha|}{|\beta|}^{-s}\frac{(|\beta|+1)^s(|\beta|+2)^s(|\alpha-\beta|+1)^s}{ |\alpha |^s  \epsilon_i \epsilon_j\epsilon_k} \\
    &\lec K^4
    b_{m+1}^2
   |\alpha|
    \underbrace{\sum_{0<k<|\alpha|}\binom{|\alpha|}{k}^{1-s}
    \frac{(k+1)^s(|\alpha|-k+1)^s}{|\alpha|}}_{\lec_s 1}
    \underbrace{\frac{(k+2)^s}{|\alpha|^s}}_{\lec_s 1}
     \\& \lec_s K^4 m b_{m+1}^2
    ,
    \label{EQ95}
    \end{split}
\end{align}
 where we used~\eqref{EQ67} in the first inequality and the binomial inequality \eqref{EQ29} in the last. Note that $\epsilon_2^{-3}$ gives $K^3$ according to the choice~\eqref{eps_rel}. 
From the estimates~\eqref{EQ29} and~\eqref{EQ67}, we deduce that 
\begin{align}
    \begin{split}
    \NL_{1,2}
    &\lec \left(\max_{|\beta|\leq|\alpha|}b_\beta(t)\right)^2|\alpha|^{-s}\tau\underbrace{\max_{|\beta|\leq|\alpha|-1}b_\beta(t)}_{\leq K}
    \\&\indeq\times
    \sum_{0<k<|\alpha|}\binom{|\alpha|}{k}^{1-s}
    (|k+1|^s|k+2|^s(|\alpha|-k+1)^s)
    \sum_{i,j,l=1,2}(\epsilon_i\epsilon_j\epsilon_l)^{-1}\tau^{-1}
    \\&\lec K
    \left(\sum_{i,j,k=1,2}(\epsilon_i\epsilon_j\epsilon_k)^{-1}\right)
   b_{m+1}^2 |\alpha|
    \underbrace{\sum_{0<k<|\alpha|}\binom{|\alpha|}{k}^{1-s}
    \frac{(k+1)^s(|\alpha|-k+1)^s}{|\alpha|}}_{\lec_s 1}
    \underbrace{\frac{(k+2)^s}{|\alpha|^s}}_{\lec_s 1}
    \\&\lec 
    \left(K \sum_{i,j,k=1,2}(\epsilon_i\epsilon_j\epsilon_k)^{-1}\right)
    b_{m+1}^2 |\alpha|
    \\& \lec K^4 m b_{m+1}^2 
    .
    \label{EQ96}
    \end{split}
\end{align}
\colb 

The treatments of $\NL_2$, $\NL_3$, and $\NL_4$ are similar, except that the combinatorial inequality is no longer needed. We have
\begin{align}
    \begin{split}
    \NL_2
    &= \frac{\epsilon^{2\alpha}\tau^{2|\alpha|-4}}{|\alpha|!^{2s}}\sum_{i,j,k=1,2}
    \left(\|D^\alpha u\|+\|D^\alpha \partial_iu\|\right)
    \left(\|\partial_j\omega\|+\|\partial_j\partial_k\omega\|\right)\|D^\alpha\omega\|
    \\&= \underbrace{\frac{\epsilon^\alpha\tau^{|\alpha|-2}}{|\alpha|!^s}\|D^\alpha\omega\|}_{= b_{\alpha}}
    \sum_{i,j,k=1,2}
    \left( \frac{\epsilon^{\alpha}\tau^{|\alpha|-2}}{|\alpha|!^s}\|D^\alpha u\|
    +\frac{\epsilon^{\alpha+e_i}\tau^{|\alpha|-1}}{|\alpha|!^s}\|D^{\alpha+e_i} u\|\frac{1}{\tau\epsilon_i}\right)
    \\&\indeq
    \times\left(\frac{\epsilon_j}{1!^s}\|\partial_j\omega\|\frac{1}{\epsilon_j}+\frac{\epsilon_j\epsilon_k}{2!^s}\|\partial_j\partial_k\omega\|\frac{2^s}{\epsilon_j\epsilon_k}\right)
    \\&\lec
    b_\alpha K^3 \sum_{i,j,k=1,2}\left(\max_{|\beta|\leq|\alpha| -1 }b_{\beta} \frac{\tau }{|\alpha|^s}\colb +\max_{|\beta|\leq|\alpha|}b_{\beta}
    \frac{\tau}{|\alpha|^s\tau\epsilon_i}\right)
    \\&\lec
    K^4b_{m+1}^2
    .
    \label{EQ97}
    \end{split}
\end{align}
We applied \eqref{EQ70} and inequality~\eqref{EQ67} in the third step of~\eqref{EQ97}.
Treating $\NL_3$ and $\NL_4$ in a similar manner, we obtain
\begin{align}
    \begin{split}
        \NL_3&=\frac{\epsilon^{2\alpha}\tau^{2|\alpha|-4}}{|\alpha|!^{2s}}
    \sum_{|\beta|=1}\binom{\alpha}{\beta}
    \|D^\beta u\|_{\lio}\|D^{\alpha-\beta}\nabla\omega\|_{\lo}\|D^\alpha\omega\|_{\lo}
    \\&\leq
    \underbrace{\frac{\epsilon^\alpha\tau^{|\alpha|-2}}{|\alpha|!^s}\|D^\alpha\omega\|_{\lo}}_{= b_\alpha }
    \sum_{|\beta|=1}\sum_{i,j,k=1,2}\left(\|D^\beta u\|_{\lo}+\|D^\beta\partial_j\partial_ku\|_{\lo}\right)
    \\&\indeq\times
    \frac{\epsilon^{\alpha-\beta+e_i}\tau^{|\alpha|-2}}{|\alpha|!^s}\|D^{\alpha-\beta+e_i}\omega\|_{\lo}\frac{\epsilon^\beta}{\epsilon_i}\binom{\alpha}{\beta}
    ,
    \end{split}
\end{align}
from where, recalling \eqref{EQ171a},  
\begin{align}
    \begin{split}
    \NL_3
    &\lec b_\alpha b_{m+1} m \left(\|D^\beta u\|_{\lo}+\frac{\epsilon^{\beta+e_j+e_k}\tau}{3!^s}
    \|D^{\beta}\partial_j\partial_k u\|_{\lo}\frac{3!^s}{\epsilon^{\beta+e_j+e_k}\tau}\right)
    \\&
    \lec b_{m+1}^2 m 
    \left(K+K^3 \tau\max_{|\beta'|\leq2}b_{\beta'}\frac{1}{\tau}\right)
    \\&
    \lec K^4 m \,b_{m+1}^2 
    ;
        \label{EQ99a}
    \end{split}
\end{align}
in the second step, we used the definition \eqref{eps_rel} of $\epsilon_2 (t)$, while in the second inequality of~\eqref{EQ99a} we used~\eqref{EQ171},\eqref{EQ171a}, \eqref{EQ70}, the order reduction estimates \eqref{EQ47}--\eqref{EQ48}, and~\eqref{EQ67}.
Next, 
\begin{align}
    \begin{split}
    \NL_4&=
    \frac{\epsilon^{2\alpha}\tau^{2|\alpha|-4}}{|\alpha|!^{2s}}
    \sum_{\substack{|\beta|=|\alpha|-1\\  \beta < \alpha }}\binom{\alpha}{\beta}
    \|D^\beta u\|_{4}
    \underbrace{\|D^{\alpha-\beta}\nabla\omega\|_{4}}
    _{\lec K}\|D^\alpha\omega\|
    \\& 
    \lec
   \underbrace{ \frac{\epsilon^\alpha\tau^{|\alpha|-2}}{|\alpha|!^s}
    \|D^\alpha\omega\|_{\lo}}_{=b_\alpha\leq b_{m+1}}
    \sum_{i, j,k =1,2}\sum_{|\beta|=|\alpha|-1}
    \left(\frac{\epsilon^\beta\tau^{|\beta|-2}}{|\beta|!^s}
    \|D^\beta u\|\frac{\tau}{|\alpha|^s}
    +\frac{\epsilon^{\beta+e_i}\tau^{|\beta|-1}}{(|\beta|+1)!^s}
    \|D^\beta\partial_iu\|\frac{1}{\epsilon_i}\right)K \epsilon^{\alpha-\beta}
    \\&
    \lec b_{m+1} K^3   \tau|\alpha|^{1-s}
    \sum_{i=1,2,3} \epsilon_{ i}^{-1}
    \\&
    \lec  K^4 b_{m+1}
    ,
    \end{split}
   \label{EQ99b}
\end{align}
where we recalled the definition~\eqref{eps_rel} of $\epsilon_2(t)$ again. Combining \eqref{EQ92}, \eqref{EQ94a}, \eqref{EQ96}, \eqref{EQ97}, \eqref{EQ99a}, and~\eqref{EQ99b}, we eventually obtain \eqref{nonlin_to_show}, as required.

\colb
\section*{Acknowledgments} 
IK and QX were supported in part by the NSF grant DMS-2205493,
while
WO was supported by the NSF grant DMS-2511556 and the Simons grant SFI-MPS-TSM-00014233. The authors gratefully acknowledge the hospitality of the Institute
for Advanced Study, where part of this work was completed during the
visit of IK and WO through the Summer Collaborators program.

\small
\medskip\medskip
\noindent
I.~Kukavica\\
{Department of Mathematics, University of Southern California, Los Angeles, CA 90089}\\
e-mail: kukavica@usc.edu

\medskip\medskip
\noindent
W.~S.~O\.za\'nski\\
{Department of Mathematics, Florida State University, Tallahassee, FL 32306}\\
{and Department of Mathematics, Princeton University, Princeton, NJ 08540}\\
e-mail: wozanski@fsu.edu

\medskip\medskip
\noindent
Q.~Xu\\
{Department of Mathematics, University of Southern California, Los Angeles, CA 90089}\\
e-mail: xuqi@usc.edu

\end{document}